\pdfoutput=1
\RequirePackage{etex}
\documentclass[a4paper]{amsart}

\usepackage[utf8]{inputenc}
\usepackage[T1]{fontenc}
\usepackage{libertine}
\usepackage[libertine]{newtxmath}
\usepackage[scr=rsfso]{mathalfa}
\usepackage{bm}

\usepackage{mathtools}
\usepackage{url}
\usepackage{hyperref}
\hypersetup{%
  draft=false,%
  pdfencoding=auto,%
  hypertexnames=false,%
  pdftitle={The wild McKay correspondence for cyclic groups of prime power order},%
  pdfauthor={Mahito Tanno and Takehiko Yasuda},%
  pdfsubject={MSC2010: Primary 14E16; Secondary 11S15, 14B05, 14E18, 14G17, 14R20.}}
\usepackage[capitalize]{cleveref}
\usepackage{autonum}
\crefformat{equation}{(#2#1#3)}
\crefrangeformat{equation}{(#3#1#4--#5#2#6)}
\crefrangeformat{align}{(#3#1#4--#5#2#6)}

\usepackage{booktabs}

\theoremstyle{plain}
\newtheorem{theorem}{Theorem}[section]
\newtheorem{proposition}[theorem]{Proposition}
\newtheorem{lemma}[theorem]{Lemma}
\newtheorem{corollary}[theorem]{Corollary}
\theoremstyle{definition}
\newtheorem{definition}[theorem]{Definition}
\newtheorem{notation}[theorem]{Notation}
\newtheorem{remark}[theorem]{Remark}
\newtheorem{example}[theorem]{Example}

\newcommand{\setRP}{\mathrm{RP}} 
\newcommand{\setC}{\mathbb{C}}

\newcommand{\setN}{\mathbb{N}}
\newcommand{\setQ}{\mathbb{Q}}
\newcommand{\setR}{\mathbb{R}}
\newcommand{\setZ}{\mathbb{Z}}
\newcommand{\rngW}[2][]{W_{#1} (#2)}
\newcommand{\rngM}{\mathcal{M}} 
\newcommand{\rngMhat}{\hat{\rngM}} 
\newcommand{\intO}{\mathcal{O}} 
\newcommand{\idlp}{\mathfrak{p}} 
\newcommand{\schA}{\mathbb{A}}
\newcommand{\schG}{\mathbb{G}}
\newcommand{\shfO}{\mathscr{O}}
\newcommand{\txtsing}{\text{\textnormal{sing}}} 
\newcommand{\mtvL}{\mathbb{L}}  
\newcommand{\Mst}{\mathnormal{M}_{\text{\textnormal{st}}}} 
\DeclareMathOperator{\rank}{rank}
\DeclareMathOperator{\Hom}{Hom}
\DeclareMathOperator{\Gal}{Gal}
\DeclareMathOperator{\ord}{ord}
\DeclareMathOperator{\Ker}{Ker}
\DeclareMathOperator{\Img}{Im}
\DeclareMathOperator{\length}{length}
\newcommand{\floor}[1]{\left\lfloor#1\right\rfloor}
\newcommand{\ceil}[1]{\left\lceil#1\right\rceil}
\DeclareMathOperator{\discrep}{discrep} 
\DeclareMathOperator{\sht}{sht}         
\DeclareMathOperator{\Spec}{Spec}
\DeclareMathOperator{\Exc}{Exc} 
\newcommand{\spcJ}[1][\infty]{\mathnormal{J}_{#1}}
\newcommand{\GCov}[1][G]{\operatorname{\mathnormal{#1}-Cov}}
\newcommand{\sepClos}[1]{{#1}^{\text{\textnormal{sep}}}} 
\newcommand{\vect}[1]{\bm{#1}}

\title{The wild McKay correspondence for cyclic groups of prime power order}

\author[M. Tanno]{Mahito Tanno}
\author[T. Yasuda]{Takehiko Yasuda}
\address{
  Department of Mathematics, Graduate School of Science, Osaka University,
  Toyonaka, Osaka 560-0043, Japan}
\email{\textsf{mahito@presche.me}}
\email{\textsf{u529757k@ecs.osaka-u.ac.jp}}
\address{
  Department of Mathematics, Graduate School of Science, Osaka University,
  Toyonaka, Osaka 560-0043, Japan}
\email{\textsf{takehikoyasuda@math.sci.osaka-u.ac.jp}}


\begin{document}
\begin{abstract}
  The \(\boldsymbol{v}\)-function is a key ingredient in the wild McKay correspondence. In
  this paper, we give a formula to compute it in terms of valuations of Witt vectors, when the
  given group is a cyclic group of prime power order.  We apply it to study singularities of a
  quotient variety by a cyclic group of prime square order.  We give a criterion whether the
  stringy motive of the quotient variety converges or not.  Furthermore, if the given
  representation is indecomposable, then we also give a simple criterion for the quotient
  variety being terminal, canonical, log canonical, and not log canonical.  With this
  criterion, we obtain more examples of quotient varieties which are klt but not
  Cohen--Macaulay.
\end{abstract}

\maketitle

\section{Introduction}\label{sec:introduction}
The subject of this paper is the wild McKay correspondence for cyclic groups of prime power
order.  The \emph{\(\bm{v}\)-function} plays an essential role in the wild McKay
correspondence; also it is considered as a common generalization of the \emph{age} invariant
in the tame McKay correspondence and of the \emph{Artin conductor}, an important invariant in
the number theory, see Wood--Yasuda~\cite{Wood2015:Mass} for details.  In spite of its
importance, it is difficult to compute the \(\bm{v}\)-functions in a general situation.  We
give an explicit formula of this function in the case of cyclic group of prime power order,
generalizing the one by the second author \cite{Yasuda2014:p-Cyclic} for the case of prime
order.  We then apply it to study the discrepancies of singularities of quotient varieties by
the cyclic group of prime square order.

The McKay correspondence relates an invariant of a representation \(V\) of a finite group
\(G\) with an invariant of the associated quotient variety \(X \coloneq V/G\).  Depending on
which type of invariant one considers, there are different approaches to the McKay
correspondence.  The one using motivic invariants originates in the works of
Batyrev~\cite{Batyrev1999:Non-Archimedean} and Denef and Loeser~\cite{Denef2002:Motivic} in
characteristic zero.  The second author~\cite{Yasuda2019:Motivic} generalized their results to
arbitrary characteristics, in particular, including the wild case, that is, the case where the
finite group in question has order divisible by the characteristic of the base field.  In what
follows, we denote by \(k\) an algebraically closed field of characteristic \(p > 0\).

\begin{theorem}[{\cite[Corollary~16.3]{Yasuda2019:Motivic}}]
  Assume that \(G\) acts on an affine space \(\schA_{k}^{d}\) linearly and effectively and
  that \(G\) has no pseudo-reflection.  Then we have
  \begin{equation}
    \Mst(X) = \int_{\GCov(D)} \mtvL^{d - \bm{v}}.
  \end{equation}
  Here \(\Mst(X)\) denotes the \emph{stringy motive} of the quotient variety \(X\),
  \(\GCov(D)\) denotes the moduli space of \emph{\(G\)-covers} of \(D \coloneq \Spec k[[t]]\),
  and \(\bm{v}\) is the \emph{\(\bm{v}\)-function} \(\bm{v} \colon \GCov(D) \to \setQ\)
  associated to the \(G\)-action on \(\schA_{k}^{d}\).
\end{theorem}

Since stringy motives contain information on singularities of the quotient variety, the above
theorem allows us to study singularities of the quotient variety \(X\) in terms of the moduli
space \(\GCov(D)\) and the \(\bm{v}\)-function on it.  For this purpose, it is important to
understand the precise structure of the moduli space and compute the \(\bm{v}\)-function.  The
second author~\cite{Yasuda2014:p-Cyclic} worked out the case \(G = \setZ/p\setZ\), using the
Artin--Schreier theory.  We extend it to the case \(G = \setZ/p^{n}\setZ\) (\(n > 0\)), using
the Artin--Schreier--Witt theory.

We find out that the \(\bm{v}\)-function in this case can be computed in terms of the
ramification jumps of the field extension corresponding to the given \(G\)-cover.  Let \(E\)
be a connected \(G\)-cover of \(D\) and \(L/k((t))\) the corresponding \(G\)-extension; the
case of connected covers is essential and the case of non-connected covers is reduced to the
case of a smaller group.  According to the Artin--Schreier--Witt theory, the extension
\(L/k((t))\) is given by an equation
\(\wp(g_{0}, g_{1}, \dotsc, g_{n - 1}) = (f_{0}, f_{1}, \dotsc, f_{n - 1})\), where
\((f_{0}, f_{1}, \dotsc, f_{n - 1}) \in \rngW[n]{k((t))}\) is a reduced Witt vector.  We can
decompose the extension \(L/k((t))\) into a tower of \(p\)-cyclic extensions
\begin{equation}
  L = K_{n - 1} \supset K_{n - 2} \supset \dotsb \supset K_{-1} = k((t)),
\end{equation}
where \(K_{i} = K_{i - 1}(g_{i})\).  Key facts here are first that the value \(\bm{v}(E)\) of
the \(\bm{v}\)-function at \(E\) is expressed in terms of ramification jumps of extensions
\(K_{i}/K_{i - 1}\) (\cref{lem:v-function-valuation}) and second that these ramification jumps
are determined by orders of \(f_i\) (\cref{prop:upper-ramification-jumps}).  We denote
\(j_{m} = - \ord f_{m}\).  Then the \((i + 1)\)-th upper ramification jump \(u_{i}\) and
\((i + 1)\)-th lower ramification jump \(l_{i}\) are given by
\begin{align}
  u_{i} &= \max\{ p^{n - 1 - m} j_{m} \mid m = 0, 1, \dotsc, i - 1\}, \\
  l_{i} &= u_{0} + (u_{1} - u_{0}) p  + \dotsb + (u_{i} - u_{i - 1}) p^{i},
\end{align}
see the proof of~\cref{thm:v-function} for details.  Since the \(\bm{v}\)-function is additive
with respect to the direct sum of representations (it is immediate from \cref{def:v-function},
see \cite[Lemma~3.4]{Wood2015:Mass}), the case of indecomposable representations is essential.
We also note that for each integer \(d \leq p^{n}\), there exists exactly one indecomposable
representation of dimension \(d\) modulo isomorphisms (see, for instance, \cite[p.~431, (64.2)
Lemma]{Curtis2006:Representation}); it corresponds to the Jordan block of size \(d\) with
eigenvalue \(1\).

\begin{theorem}[\cref{thm:v-function}]
  Assume that the \(G\)-representation \(V\) is indecomposable of dimension \(d\).  With the
  notation as above, we have
  \begin{equation}\label{formula:v-func}\tag{\(\star\)}
    \bm{v}(E) = \sum_{\substack{
        0 \leq i_{0} + p i_{1} + \dotsb + p^{n - 1} i_{n - 1} < d, \\
        0 \leq i_{0}, i_{1}, \dotsb, i_{n - 1} < p
      }}
    \ceil{\frac{
        i_{0} p^{n - 1} + i_{1} p^{n - 2} l_{1} + \dotsb + i_{n - 1} l_{n_{1}}
      }{p^{n}}}.
  \end{equation}
\end{theorem}

In relation to the minimal model program, it is natural to ask: how can we determine
representation-theoretically when a quotient variety \(V/G\) (with \(G\) an arbitrary finite
group) is terminal, canonical, log terminal or log canonical?  From the wild McKay
correspondence with the formula~\cref{formula:v-func}, we can give a partial answer to this
question.  Note that the formula~\cref{formula:v-func} implicitly includes many maxima so that
we have to make a case-by-case analysis to compute the integral
\(\int_{\GCov(D)} \mtvL^{d - \bm{v}}\). Thus the computation rapidly becomes harder, as the
exponent \(n\) increases.  For this reason, we focus on the case \(n=2\) to evaluate the
integral and get some results on singularities. Before stating our results in this direction,
we need to introduce some invariants of representations.  For an indecomposable representation
\(V\) of \(G= \setZ/p^{2}\setZ\) of dimension \(d\), writing \(d = qp+r\) (\(0 \le r < p \)),
we define
\begin{align}
  B_{V} &= \frac{qp(q-1)}{2}+qr, \\
  C_{V}^{>} &= p\left(\frac{qp(p-1)}{2}+\frac{r(r-1)}{2}\right) + (p^2-p+1) \left(\frac{qp(q-1)}{2}+qr\right).
\end{align}
We generalize them to decomposable representations in the way that they become additive for
direct sums.

\begin{theorem}[{\cref{thm:cond-converge}}]
  Assume that \(G = \setZ/p^{2}\setZ\).  The integral \(\int_{\GCov(D)} \mtvL^{d - \bm{v}}\)
  converges if and only if the following inequalities hold:
  \begin{align}
    B_{V} &\geq p, \\
    C_{V}^{>} &\geq p^{3} - p + 1.
  \end{align}
\end{theorem}

From the wild McKay correspondence, the convergence of the integral
\(\int_{\GCov(D)} \mtvL^{d - \bm{v}}\) is equivalent to that of the stringy motive
\(\Mst(X)\). The latter implies that \(X\) is log terminal and the converse holds if the pair
has a log resolution. Thus we obtain the following corollary:

\begin{corollary}[\cref{prop:cond-klt}]
  Assume that \(G = \setZ/p^{2}\setZ\) has no pseudo-reflection.  If the inequalities
  \(B_{V} \geq p\) and \(C_{V}^{>} \geq p^{3} - p + 1\) hold, then the quotient variety
  \(X = V/G\) is log terminal.  Furthermore, if there exists a log resolution of \(X\), then
  the converse is also true.
\end{corollary}

Furthermore, for an indecomposable \(\setZ/p^{2}\setZ\)-representation \(V\), we can estimate
the discrepancies/total discrepancy of the quotient variety:

\begin{theorem}[\cref{thm:estimate-discrep}]
  Assume that \(G = \setZ/p^{2}\setZ\) and \(V\) is an indecomposable \(G\)-representation of
  dimension \(d\) (\(p + 1 < d \leq p^{2}\)).  Then,
  \begin{equation}
    \text{\(X\) is }
    \begin{cases}
      \text{terminal}, \\
      \text{canonical}, \\
      \text{log canonical}, \\
      \text{not log canonical}
    \end{cases}
    \text{if and only if }
    \begin{cases}
      d \geq 2p + 1, \\
      d \geq 2p, \\
      d \geq 2p - 1, \\
      d < 2p - 1.
    \end{cases}
  \end{equation}
\end{theorem}

We note that the indecomposable representation of \(\setZ/p^{2}\setZ\) of dimension \(d\) is
not effective if \(d \le p\), has pseudo-reflections if \(d = p + 1\) and does not have a
pseudo-reflection if \(d > p + 1 \).

Related to the minimal model program in positive characteristics, some singularities which are
klt but not Cohen--Macaulay are constructed in recent years (see
Kov\'{a}cs~\cite{Kovacs2018:Non-Cohen},
Yasuda~\cite{Yasuda2014:p-Cyclic,Yasuda2019:Discrepancies},
Cascini--Tanaka~\cite{Cascini2019:Purely}, Bernasconi~\cite{Bernasconi2021:Kawamata-Viehweg},
Arvidsson--Bernasconi--Lacini~\cite{Arvidsson2020:Kawamata-Viehweg}, and
Totaro~\cite{Totaro2019:Failure}; see also \cite{Yasuda2018:Cohen}).  The theorem above
provides more such examples; for instance, if \(V\) is the indecomposable
\(\setZ/4\setZ\)-representation of dimension \(4\) in characteristic \(2\), the quotient
variety \(V/(\setZ/4\setZ)\) is canonical but not Cohen--Macaulay.

We now give a few comments on the case \(G\) has pseudo-reflections.  Generally, if a finite
group \(G\) has a pseudo-reflection, then we can find a \(\setQ\)-Weil divisor \(\Delta\) on
\(X = V/G\) such that \(V \to (X, \Delta)\) is crepant.  The wild McKay correspondence theorem
holds for log pairs by replacing \(\Mst(X)\) by \(\Mst(X, \Delta)\).  For a representation of
\(G = \setZ/p^{n}\setZ\) with general \(n\), we determine when there is a pseudo-reflection
(\cref{cor:pseudo-ref}).  Moreover we show that if the given effective \(G\)-representation
has a pseudo-reflection, then the divisor \(\Delta\) as above on the quotient variety \(X\) is
irreducible and has multiplicity \(p - 1\) (\cref{prop:boundary}), hence the pair
\((X, \Delta)\) is not log canonical unless \(p = 2\).  If \(p = n = 2\), then whether or not
the pair is log canonical depends on whether the representation has a direct summand of
dimension one (\cref{rem:not-lc:p=2}).

We also note that although we work over an algebraically closed field throughout the paper for
the simplicity reason, it is straightforward to generalize our results to any field of
characteristic \(p > 0\) simply by the base change.

The outline of this paper is as follows.  In \cref{sec:G-covers}, we first recall basic facts
about the Artin--Schreier--Witt theory.  After that, we describe the moduli space
\(\GCov(D)\) of \(G\)-covers of \(D = \Spec k[[t]]\) and decompose it to strata \(\GCov(D;
\vect{j})\).  In \cref{sec:v-function}, we see that \(\bm{v}\)-functions are written by
valuations of Witt vectors and by upper/lower ramification jumps of \(G\)-extensions.  In
\cref{sec:discrepancy}, we briefly review the wild McKay correspondence and its application to
singularities.  In \cref{sec:case:n=2}, we discuss the case \(G = \setZ/p^{2}\setZ\) and give
our main results as corollaries of \cref{thm:v-function}.

\subsection*{Notation and convention}
Unless otherwise noted, we follow the following notation.  We denote by \(k\) an algebraically
closed field of characteristic \(p > 0\) and by \(K = k((t))\) the field of formal Laurent
power series over \(k\).  We set \(G = \langle \sigma \rangle\) a cyclic group of order
\(p^{n}\).

\subsection*{Acknowledgments}
We would like to thank Takeshi Saito for giving us useful information. This work was supported
by JSPS KAKENHI Grant Numbers 18H01112 and 18K18710.

\section{\texorpdfstring{\(G\)}{G}-covers of the formal punctured disk}\label{sec:G-covers}
In this section, we discuss the moduli spaces of \'{e}tale \(G\)-covers of the formal
punctured disk \(D^{*} \coloneq \Spec K\).

By an \'{e}tale \(G\)-cover \(E^{*} \to D^{*}\) of \(D^{*}\), we mean it is a finite \'{e}tale
morphism of degree \(\# G\) endowed with a \(G\)-action on \(E^{*}\) such that
\(E^{*} / G = D^{*}\).  By a \emph{\(G\)-cover} \(E \to D \coloneq \Spec k[[t]]\), we mean it
is the normalization \(E\) of \(D\) in an \'{e}tale \(G\)-cover \(E^{*} \to D^{*}\).  We
denote by \(\GCov(D^{*})\) (resp.\ \(\GCov(D)\)) the set of all \'{e}tale \(G\)-covers of
\(D^{*}\) (resp.\ \(G\)-covers of \(D\)).  Since there is a one-to-one correspondence between
\(\GCov(D^{*})\) and \(\GCov(D)\), we sometimes identify them.

\subsection{The Artin--Schreier--Witt theory}\label{sec:ASW-theory}
Let us recall some basic facts from the theory of Witt vectors.  We denote by \(\rngW[m]{K}\)
the ring of Witt vectors of length \(m\) over \(K\).  We introduce important morphisms.  One
is the \emph{Frobenius} morphism
\begin{equation}
  \operatorname{Frob} \colon \rngW[m]{K} \to \rngW[n]{K},
  (a_{0}, a_{1}, \dotsc) \mapsto (a_{0}^{p}, a_{1}^{p}, \dotsc).
\end{equation}
We denote by \(\wp \coloneq \operatorname{Frob} - \operatorname{id}\) the Artin--Schreier
morphism.  The other is the \emph{Verschiebung} morphism
\begin{equation}
  \rngW[m]{K} \to \rngW[m + 1]{K},
  (a_{0}, a_{1}, \dotsc) \mapsto (0, a_{0}, a_{1}, \dotsc).
\end{equation}
They are homomorphisms of additive groups.  Note that the Verschiebung morphism commutes with
\(\wp\) and that
\begin{equation}
  (a_{0}, a_{1}, \dotsc, a_{l - 1}, a_{l}, \dotsc)
  = (a_{0}, a_{1}, \dotsc, a_{l - 1}, 0, \dotsc) + (0, \dots, 0, a_{l}, a_{l + 1}, \dotsc)
\end{equation}
holds for every \(l \geq 1\).

Let us denote by \(\sepClos{K}\) the separable closure of \(K\), by \(K_{p^{n}}\) the maximal
abelian extension of exponent \(p^{n}\) over \(K\).  As sets, we can describe
\begin{align}
  \GCov(D^{*}) &= \Hom_{\text{cont}}(\Gal(\sepClos{K} / K), \setZ / {p^{n} \setZ}) \\
               &= \Hom_{\text{cont}}(\Gal(K_{p^{n}} / K), \setZ / {p^{n} \setZ}) \\
               &= \Hom_{\text{cont}}(\Gal(K_{p^{n}} / K), {\tfrac{1}{p^{n}} \setZ} / \setZ) \\
               &\overset{(\heartsuit)}{=}
                 \Hom_{\text{cont}}(\Gal(K_{p^{n}} / K), \setQ / \setZ) \\
               &\overset{(\clubsuit)}{=}
                 \rngW[n]{K} / \wp(\rngW[n]{K}).
\end{align}
Since every element of \(\Gal(K_{p^{n}} / K)\) has order dividing \(p^{n}\), the image of any
morphisms \(\Gal(K_{p^{n}}/ K) \to \setQ / \setZ\) is contained in \((1/p^{n}) \setZ / \setZ\)
and hence the equality \((\heartsuit)\) holds (see
\cite[pp.~340--341]{Neukirch2008:Cohomology} for details).  The equality \((\clubsuit)\) is a
consequence of~\cite[Theorem~6.1.9]{Neukirch2008:Cohomology}.  For a Witt vector
\(\vect{f} \in \rngW[n]{K}\), we denote by \(E_{\vect{f}}^{*}\) the \(G\)-cover of \(D^{*}\)
corresponding to the class of \(\vect{f}\).  Note that \(E_{\vect{f}}^{*}\) is connected if
and only if \(f_{0} \notin \wp(K)\).  More explicitly, we can see the following (see, for
instance, \cite[Chapter~VI, Exercise~50]{Lang2002:Algebra}). 

\begin{proposition}\label{prop:cyclic-ext}
  For a Galois extension \(L/K\), it is \(p^{n}\)-cyclic if and only if there
  exists a Witt vector \(\vect{f} = (f_{0}, f_{1}, \dotsc, f_{n-1}) \in
  \rngW[n]{K}\) with \(f_{0} \notin \wp(K)\) such that \(L = K(g_{0}, g_{1}, \dots, g_{n-1})\)
  where the Witt vector \(\vect{g} = (g_{0}, g_{1}, \dotsc, g_{n-1})\) is a root of an
  equation \(\wp(\vect{g}) = \vect{f}\).  Moreover, a  generator \(\sigma\) of the Galois
  group \(\Gal(L/K)\) is given by \(\sigma(\vect{g}) = \vect{g} + \mathbf{1}\).
\end{proposition}

We next find good representatives of elements of \(\rngW[n]{K}/\wp(\rngW[n]{K})\).

\begin{notation}
  We put \(\setN' \coloneq \{j \in \setZ \mid j > 0, p \nmid j\}\).
\end{notation}

\begin{lemma}\label{Yas14:lem2.3}
  For \(f \in K\), there exists a unique Laurent polynomial of the form
  \begin{equation}
  g = \sum_{i \in \setN'} g_{-i}t^{-i} \in k[t^{-1}] \subset K
  \end{equation}
  such that \(f - g \in \wp(K)\).
\end{lemma}
\begin{proof}
  See, for instance, \cite[Lemma~2.3]{Yasuda2014:p-Cyclic}. 
\end{proof}

We call a Laurent polynomial of the above form a \emph{representative polynomial}.  We denote
by \(\setRP_{k}\) the set of representative polynomials.  We can extend \cref{Yas14:lem2.3} as
follows.
\begin{lemma}\label{lem:representative-Witt-vector}
  For a Witt vector \(\vect{f} \in \rngW[n]{K}\), there exists a unique
  \(\vect{g} = (g_{0}, g_{1}, \dotsc, g_{n - 1}) \in \rngW[n]{K}\) such that each \(g_{l}\)
  is a representative polynomial and \(\vect{f} - \vect{g} \in \wp(\rngW[n]{K})\).
\end{lemma}
\begin{proof}
  We prove by induction on \(n\).  The case \(n = 1\) is just
  \cref{Yas14:lem2.3}.  Let us denote by \(f_{l}\) the \(l\)-th component of \(\vect{f}\).
  Take \(h_{0} \in K\) satisfying \(g_{0} = f_{0} + \wp(h_{0})\), where \(g_{0}\) is the
  unique representative polynomial.  In the Witt ring \(\rngW[m]{K}\), we have
  \begin{align}
    \label{eq:key-lemma:reprWittVec}
    (f_{0}, \dotsc) + \wp(h_{0}, \dotsc) &= (f_{0} + \wp(h_{0}), \dotsc) \\
    &= (g_{0}, \dotsc).
  \end{align}
  Without loss of generality, we may assume that \(\vect{f} = (g_{0}, f_{1}, \dotsc)\).
  From the induction hypothesis, there exists \(g_{1}, g_{2}, \dotsc, g_{n - 1}\) uniquely
  such that each \(g_{l}\) is a representative polynomial and
  \begin{equation}
    (f_{1}, \dotsc, f_{n - 1}) \equiv (g_{1}, \dotsc, g_{n - 1}) \pmod{\wp(\rngW[n - 1]{K})}.
  \end{equation}
  holds.  Since the Verschiebung morphism commutes with \(\wp\), thus we have
  \begin{equation}
    (0, f_{1}, \dotsc, f_{n - 1}) \equiv (0, g_{1}, \dotsc, g_{n - 1}) \pmod{\wp(\rngW[n]{K})}.
  \end{equation}
  Then
  \begin{align}
    (g_{0}, f_{1}, \dotsc, f_{n-1})
    & = (g_{0}, 0, \dotsc, 0) + (0, f_{1}, \dotsc, f_{n-1}) \\
    &\equiv (g_{0}, 0, \dotsc, 0) + (0, g_{1}, \dotsc, g_{n-1}) \pmod{\wp(\rngW[n]{K})} \\
    &= (g_{0}, g_{1}, \dotsc, g_{n-1}).
  \end{align}
  The first and last equality follows from the property of the Verschiebung morphism.
  Therefore, we have proved the existence of \(\vect{g}\).

  Next, we show the uniqueness.  \cref{eq:key-lemma:reprWittVec} shows that the first entry
  \(g_{0}\) is uniquely determined.  Suppose that \((g_{0}, g_{1}, \dotsc, g_{n - 1})\) and
  \((g_{0}, g_{1}', \dotsc, g_{n - 1}')\) satisfy the condition.  Then we have
  \begin{equation}
    (0, f_{1}, \dotsc, f_{n - 1}) \equiv (0, g_{1}, \dotsc, g_{n - 1})
    \equiv (0, g_{1}', \dotsc, g_{n - 1}') \pmod{\wp(\rngW[n]{K})}.
  \end{equation}
  Again from the induction hypothesis, this shows that \(g_{1}, \dotsc, g_{n - 1}\) are
  uniquely determined.
\end{proof}

We call a Witt vector \(\vect{g} = (g_{0}, g_{1}, \dotsc, g_{n - 1})\) consisting of
representative polynomials \(g_{l}\) is a \emph{representative Witt vector}.  More generally,
a Witt vector \(\vect{f} = {(f_{l})}_{l} \in \rngW[n]{K}\) is called \emph{reduced} (or
\emph{standard form}) if \(p \nmid v_{K}(f_{l})\) and \(v_{K}(f_{l}) < 0\) for every \(l\),
where \(v_{K}\) denotes the normalized valuation on \(k\).

\begin{corollary}
  We have a one-to-one correspondence
  \begin{equation}
    \GCov(D^{*}) \leftrightarrow {\left(\setRP_{k}\right)}^{n}.
  \end{equation}
\end{corollary}

\begin{remark}\label{rem:moduli}
  The corollary shows that \(\GCov(D^{*})\) is identified with the \(k\)-point set of the
  ind-scheme \(\schA^{\infty}_k:=\varinjlim_{n \in \setN} \schA^{n}_k \), where the transition
  map \(\schA^{n}_k \to \schA^{n+1}_k\) is the standard closed embedding. In fact, the coarse
  moduli space of \(\GCov(D^{*})\) is the inductive perfection (that is, the inductive limit
  with respect to Frobenius morphisms) of this space \(\schA^{\infty}_k\),
  see~\cite{Harbater1980:Moduli}.  To get the fine moduli stack, we further need to take the
  product of it with the stack \(\mathrm{B}G\), see \cite{Tonini2017:Moduli}.
\end{remark}

\subsection{Stratification and parameterization}\label{sec:strat-pram}
In what follows, we follow the convention that \(\ord 0 = \infty\).  For a Witt vector
\(\vect{f} = {(f_{l})}_{l} \in \rngW[n]{K}\), we denote the vector
\(\ord \vect{f} \coloneq{ (\ord f_{l})}_{l}\).  When \(E^{*}\) is a \(G\)-cover of \(D^{*}\)
corresponding to the representative Witt vector \(\vect{f}\), we denote
\(\ord E^{*} = \ord \vect{f}\).

\begin{definition}
  For an \(n\)-tuple \(\vect{j} ={(j_{l})}_{l} \in {(\setN' \cup \{- \infty\})}^{n}\), set
  \(- \vect{j} = {(-j_{l})}_{l}\).  We define
  \begin{align}
    \GCov(D^{*}; \vect{j}) &\coloneq
      \left\{E^{*} \in \GCov(D^{*}) \mid \ord E^{*} = - \vect{j}\right\}, \\
    \setRP_{k, \vect{j}} &\coloneq
      \prod_{l=0}^{n-1} \left\{f \in \setRP_{k} \mid \ord f = - j_{l}\right\}
  \end{align}
  For the case \(\vect{j} = (j_{0})\), we write \(\setRP_{k, j_{0}}\) in stead of
  \(\setRP_{k, (j_{0})}\).
\end{definition}

\begin{remark}
  We remark that we consider \(\setN' \cup \{0\}\) instead of \(\setN' \cup \{- \infty\}\) in
  the previous paper~\cite{Yasuda2014:p-Cyclic}.  However, our convention in the present paper
  is more suitable for computation below.
\end{remark}

When \(n = 1\), we have the following one-to-one correspondences
(see~\cite[10, Proposition~2.11]{Yasuda2014:p-Cyclic})
\begin{equation}
  \GCov(D^{*}; j) \leftrightarrow \setRP_{k, j}
    \leftrightarrow k^{\times} \times k^{j - 1 - \floor{j/p}}.
\end{equation}

Here \(\floor{\bullet}\) denotes the floor function, which assigns a real number \(a\) to the
greatest integer \(\floor{a}\) less than or equal to \(a\).  When \(j = - \infty \), the space
\(\GCov(D^{*}; -\infty) \) is a point.  The following is straightforward.

\begin{proposition}
  For \(\vect{j} = {(j_{l})}_{l} \in {(\setN'\cup \{-\infty\})}^{n}\),
  we have one-to-one correspondences
  \begin{equation}
    \GCov(D^{*}; \vect{j}) \leftrightarrow
    \setRP_{k, \vect{j}} \leftrightarrow
    \prod_{j_l\ne -\infty} \left(k^{\times} \times k^{j_{l} - 1 - \floor{j_{l}/p}}\right).
  \end{equation}
\end{proposition}

We now regard \(k^{\times} \times k^{n}\) as the variety
\(\schG_{m, k} \times \schA_{k}^{n}\).  Then the above correspondence gives a structure of
variety to \(\GCov(D^{*}; \vect{j})\).  Thus, \(\GCov(D^{*}; \vect{j})\) can be thought of as
an infinite-dimensional space admitting the stratification
\begin{equation}
  \GCov(D^{*}) = \coprod_{\vect{j}} \GCov(D^{*}; \vect{j})
\end{equation}
into countable finite-dimensional strata.

\begin{remark}
  Varieties \(\schG_{m, k} \times \schA_{k}^{n}\) are neither fine or coarse moduli spaces of
  \(G\)-covers (see \cref{rem:moduli}).  However we can construct families of \(G\)-covers
  over these spaces in a similar way as in \cite[Section~2.4]{Yasuda2014:p-Cyclic} and get
  morphisms from these spaces to the corresponding fine moduli stacks which are bijective on
  geometric points.  Thus, as justified in \cite{Tonini2019:Moduli}, we can use the above
  varieties as our parameter spaces of \(G\)-covers in our context of motivic integration.
\end{remark}

\subsection{Explicit description of \texorpdfstring{\(G\)}{G}-actions on
  \texorpdfstring{\(G\)}{G}-covers}\label{sec:G-actions}
Let \(\vect{f} = {(f_{l})}_{l} \in \setRP_{k, \vect{j}}\) be a representative Witt vector of
order \(\ord \vect{f} = - \vect{j}\) and \(\vect{g} = {(g_{l})}_{l}\) a root of
\(\wp(\vect{g}) = \vect{f}\).  We assume that the extension \(L = K(\vect{g})\) is a
\(G\)-extension of \(K\) and that the generator \(\sigma\) of \(G\) acts on \(L\) by
\(\sigma(\vect{g}) = \vect{g} + \vect{1}\).  We can decompose the extension \(L/K\) into a
tower of \(p\)-cyclic extensions
\begin{equation}
  L = K_{n-1} \supset K_{n-2} \supset \dotsb \supset K_{0} \supset K_{-1} = K
\end{equation}
where \(K_{i} = K_{i-1}(g_{i})\).  Indeed, \(\sigma^{p^{i}}|_{K_{i}}\) fixes \(K_{i-1}\) and
its order is \(p\).  For each extension \(K_{i}/K_{i-1}\), \(g_{i}\) is a root of an equation
\begin{equation}
  g_{i}^{p} - g_{i} + (\text{polynomial in \(g_{0}, g_{1}, \dotsc, g_{i-1}\)}) = f_{i}.
\end{equation}

We denote by \(v_{K_{i}}\) the normalized valuation on \(K_{i}\).  For each \(i\), there
exists an \(h_{i} \in K_{i-1}\) such that
\(\tilde{f}_{i} = {(g_{i} + h_{i})}^{p} - (g_{i} + h_{i})\),
\(p \nmid v_{K_{i-1}}(\tilde{f}_{i})\) and \(v_{K_{i - 1}}(\tilde{f}_{i}) < 0\).  We set
\(\tilde{g}_{i} \coloneq g_{i} + h_{i}\) and
\(\tilde{\vect{g}} \coloneq {(\tilde{g}_{l})}_{l}\).  Since \(\tilde{g}_{l}^{i_{l}}\)
(\(0 \leq i_{l} < p\)) form a basis of \(K_{l} / K_{l - 1}\), thus
\(\tilde{g}_{0}^{i_{0}} \tilde{g}_{1}^{i_{1}} \dotsm \tilde{g}_{n-1}^{i_{n-1}}\)
(\(0 \leq i_{0}, i_{1}, \dotsc , i_{n-1} < p\)) form a basis of \(L/K\).

\begin{notation}
  For a \(k\)-algebra \(M\) endowed with a \(G\)-action, we denote
  \(\delta \coloneq \sigma - \mathrm{id}_{M}\) a \(k\)-linear operator.
  For \(d \in \setZ_{\geq 0}\),
  we write \(M^{\delta^{d} = 0} \coloneq \Ker (\delta^{d} \colon M \to M)\).
\end{notation}

For an \(n\)-tuple \(I = (i_{0}, i_{1}, \dotsc, i_{n-1}) \in {\{0, 1, \dotsc, p-1\}}^{n}\), we
use a multi-index notation
\(\tilde{\vect{g}}^{I} = \tilde{g}_{0}^{i_{0}} \tilde{g}_{1}^{i_{1}} \dotsm %
\tilde{g}_{n-1}^{i_{n-1}}\).  We remark that to give an \(n\)-tuple \(I = {(i_{l})}_{l}\) is
equivalent to give an integer \(a_{I} = \sum_{l=0}^{n-1} i_{l} p^{l}\).

\begin{proposition}\label{prop:basis-ker-delta}
  For any integer \(a_{I}\) with \(1 \leq a_{I} < p^{n}\) and for any \(h \in K\), we have
  \(\delta^{a_{I}}(\tilde{\vect{g}}^{I} h) \in k^{\times} \cdot h\) and
  \(\delta^{a_{I} + 1}(\tilde{\vect{g}}^{I} h) = 0\).  Therefore, for each integer \(d\) with
  \(0 \leq d \leq p^{n}\), we have
  \begin{equation}
    L^{\delta^{d} = 0} = \bigoplus_{a_{I} = 0}^{d-1} K \cdot \tilde{\vect{g}}^{I}.
  \end{equation}
\end{proposition}
\begin{proof}
  The case \(n = 1\) is~\cite[Lemma~2.15]{Yasuda2014:p-Cyclic}.  By direct computation, we get
  \(\delta^{p^{m}} = \sigma^{p^{m}} - \mathrm{id}\) for \(0 \leq m \leq n\).  The
  Artin--Schreier--Witt theory says that \(\sigma^{p^{m}}\) fixes the subfield
  \(K_{m - 1} = K(\tilde{g}_{0}, \tilde{g}_{1} , \dotsc, \tilde{g}_{m-1})\) and that
  \(\sigma^{p^{m}}(\tilde{g}_{m}) = \tilde{g}_{m} + 1\).  Furthermore, \(\delta^{p^{m}}\) is
  not only \(k\)-linear but also \(K_{m-1}\)-linear.  For \(1 \leq i_{m} < p\),
  \begin{align}
    \delta^{p^{m}}(\tilde{g}_{m}^{i_{m}})
      &= {(\tilde{g}_{m} + 1)}^{i_{m}} - \tilde{g}_{m}^{i_{m}} \\
      &= i_{m} \tilde{g}_{m}^{i_{m}-1} + \dotsb + i_{m} \tilde{g}_{m} + 1.
  \end{align}
  Applying \({(\delta^{p^{m}})}^{i_{m} - 1}\) to this, by the induction on \(i_{m}\), we get
  \begin{equation}
    {(\delta^{p^{m}})}^{i_{m}}(\tilde{g}_{m}^{i_{m}})
      = i_{m} \cdot {(\delta^{p^{m}})}^{i_{m}-1}(\tilde{g}_{m}^{i_{m}-1})
  \end{equation}
  and hence \({(\delta^{p^{m}})}^{i_{m}}(\tilde{g}_{m}^{i_{m}}) = i_{m} !\). 
  Then we have
  \begin{align}
    \delta^{i_{0} + i_{1}p + \dotsb + i_{n}p^{n}}\left(\tilde{g}_{0}^{i_{0}}
      \tilde{g}_{1}^{i_{1}} \dotsm \tilde{g}_{n}^{i_{n}} h\right)
    &= \delta^{i_{0} + \dotsb + i_{n-1}p^{n-1}}\left( {\left(\delta^{p^{n}}\right)}^{i_{n}}
      \left(
      \tilde{g}_{0}^{i_{0}} \dotsm \tilde{g}_{n-1}^{i_{n-1}} \tilde{g}_{n}^{i_{n}} h\right)
      \right) \\
    &= \delta^{i_{0} + \dotsb + i_{n-1}p^{n-1}}\left(
      \tilde{g}_{0}^{i_{0}} \dotsm \tilde{g}_{n-1}^{i_{n-1}} \cdot i_{n}! \, h 
      \right),
  \end{align}
  and the first assertion follows from the induction on \(n\).

  It is clear that
  \(L^{\delta^{p^{n}} = 0} = L = \bigoplus_{a_{I} = 0}^{p^{n} - 1} K \cdot
  \tilde{\vect{g}}^{I}\) holds.  Assume
  \(x = \sum_{I} x_{I} \tilde{\vect{g}}^{I} \in L^{\delta^{d - 1} = 0}\).  Since
  \(L^{\delta^{d - 1} = 0} \subset L^{\delta^{d} = 0}\), thus we have \(x_{I} = 0\) for
  \(a_{I} \geq d\) by the induction on \(d\).  From the first assertion, we have
  \(\delta^{d - 1}(x) = \delta^{d - 1}(x_{J} \tilde{\vect{g}}^{J}) = 0\), where \(J\) is the
  index satisfying \(a_{J} = d - 1\).  Again from the first assertion, this shows
  \(x_{J} = 0\).  Therefore, we have
  \(L^{\delta^{d - 1} = 0} \subset \bigoplus_{a_{I} = 0}^{d - 2} K \cdot
  \tilde{\vect{g}}^{I}\).  The converse also follows form the first assertion.
\end{proof}

\begin{corollary}\label{cor:basis-of-ker-delta}
  We denote by \(\intO_{L}\) the integer ring of \(L\) and \(v_{L}\) the normalized valuation
  on \(L\).  For an \(n\)-tuple
  \(I = (i_{0}, i_{1}, \dotsc, i_{n-1}) \in {\{0, 1, \dotsc, p-1\}}^{n}\),
  we put \(n_{I} \coloneq \ceil{- v_{L}(\tilde{\vect{g}}^{I})/{p^{n}}}\).  Here
  \(\ceil{\bullet}\) denotes the ceiling function, which assigns a real number \(a\) to the
  least integer \(\ceil{a}\) greater than or equal to \(a\).  Then we have
  \begin{equation}
    \intO_{L}
      = \prod_{v_{L}(\tilde{\vect{g}}^{I} t^{n}) \geq 0} k \cdot \tilde{\vect{g}}^{I} t^{n}
      = \bigoplus_{I} k[[t]] \cdot \tilde{\vect{g}}^{I} t^{n_{I}}.
  \end{equation}
  Moreover, for each integer \(d\) with \(0 \leq d \leq p^{n}\), we have
  \begin{equation}
    \intO_{L}^{\delta^{d}=0}
      = \prod_{\substack{v_{L}(\tilde{\vect{g}}^{I} t^{n}) \leq 0, \\%
                                      0 \leq a_{I} < d}} k \cdot \tilde{\vect{g}}^{I} t^{n}
      = \bigoplus_{0 \leq a_{I} < d} k[[t]] \cdot \tilde{\vect{g}}^{a_{I}} t^{n_{I}}.
  \end{equation}
\end{corollary}
\begin{proof}
  By definition, \(v_{K_{l}}(\tilde{g}_{l}^{i_{l}})\) takes distinct values modulo \(p\) when
  \(i_{l}\) runs from \(0\) to \(p - 1\).  Therefore, \(v_{L}(\tilde{\vect{g}}^{I})\) takes
  distinct values modulo \(p^{n}\) when \(a_{I}\) runs from \(0\) to \(p^{n} - 1\).  This
  proves the first assertion.  The second assertion follows from \cref{prop:basis-ker-delta}
  and the first assertion.
\end{proof}

\section{\texorpdfstring{\(\bm{v}\)}{v}-functions}\label{sec:v-function}
Suppose that we are given a faithful \(k\)-linear action of \(G\) on \(\schA_{k}^{d}\).  Let
\(E\) be a \(G\)-cover of \(D = \Spec k[[t]]\), which means the normalization of \(D\) in a
\(G\)-cover \(E^{*}\) of \(D^{*} = \Spec K\).  Let \(\shfO_{E}\) be the coordinate ring of
\(E\).  Then the direct sum \(\shfO_{E}^{\oplus d}\) has two \(G\)-actions.  One is the
diagonal action induced from the \(G\)-action on \(\shfO_{E}\).  The other is given by the
composition \(G \to GL(d, k) \hookrightarrow GL(d, \shfO_{E})\), where the left map is
associated to the \(G\)-action on \(\schA_{k}^{d}\).

\begin{definition}\label{def:tuning-module}
  We define the \emph{tuning module} \(\Xi_{E} \subset \shfO_{E}^{\oplus d}\) to be the
  submodule of elements on which the two actions above coincide.
\end{definition}

\begin{lemma}[{\cite[Proposition~6.3]{Yasuda2017:Toward}}]
  The tuning module \(\Xi_{E}\) is a free \(k[[t]]\)-module of rank \(d\).
\end{lemma}

\begin{definition}\label{def:v-function}
  We define the \(\bm{v}\)-function \(\bm{v} \colon \GCov(D) \to \setQ, E \mapsto \bm{v}(E)\)
  as follows.  Let \(x_{i} = {(x_{ij})}_{1 \leq j \leq d} \in \shfO_{E}^{\oplus d}\)
  (\(1 \leq i \leq d\)) be a \(k[[t]]\)-basis of \(\Xi_{E}\).  Then we define
  \begin{align}
    \bm{v}(E)
    &= \frac{1}{\# G} \length \frac{\shfO_{E}}{(\det {(x_{ij})}_{i, j})} \\
    &= \frac{1}{\# G}
       \length \frac{\shfO_{E}^{\oplus d}}{\shfO_{E} \cdot \Xi_{E}}.
  \end{align}
\end{definition}

By abuse of notation, we write \(\Xi_{E^{*}} = \Xi_{E}\) and \(\bm{v}(E^{*}) = \bm{v}(E)\),
because \(E\) is the normalization of \(D\) in \(E^{*}\).

The \(\bm{v}\)-function depends on the given \(G\)-representation. We sometimes write the
\(\bm{v}\)-function as \(\bm{v}_{V}\), referring to the representation \(V\) in question.  If
\(E\) is connected and \(v\) denotes the normalized valuation on \(\shfO_{E}\), then we have
\begin{equation}
  \bm{v}(E) = \frac{1}{\#G} v\left(\det{(x_{ij})}_{i,j}\right).
\end{equation}
When \(E\) is not connected and \(E'\) is a connected component with the stabilizer subgroup
\(H \subset G\), then we have
\begin{equation}
  \bm{v}_{V}(E) = \bm{v}_{W}(E'),
\end{equation}
where \(W\) is the restriction of \(V\) to \(H\).

\subsection{The indecomposable case}
Let \(V\) be an indecomposable \(G\)-representation of dimension \(d\).  Since
\(\bm{v}_{V \oplus W} = \bm{v}_{V} + \bm{v}_{W}\) holds, thus the case of indecomposable
representations is essential.  We denote the coordinate ring of the affine space \(V\) by
\(k[\vect{x}] = k[x_{1}, x_{2}, \dotsc, x_{d}]\).  We choose coordinates so that the chosen
generator \(\sigma\) of \(G\) acts by
\begin{equation}
  x_{i} \mapsto
  \begin{cases}
    x_{i} + x_{i+1} & (i \neq d) \\
    x_{d} & (i = d).
  \end{cases}
\end{equation}
It amounts to taking the Jordan standard form of \(\sigma\).  We have \(d \leq p^{n}\), since
the order of a Jordan block of size \(m\) with eigenvalue \(1\) is the greatest power of \(p\)
does not exceeding \(m\).  Let \(E^{*}\) be a \(G\)-cover of \(D^{*} = \Spec K\).  We
also assume that \(E^{*} = \Spec L\), where \(L/K\) is a \(G\)-extension.  With the
notation of \cref{sec:G-actions}, the tuning module \(\Xi_{E^{*}}\) of \(E^{*}\) is written as
\begin{align}
  \Xi_{E^{*}} &= \left\{ (\alpha_{1}, \alpha_{2}, \dots, \alpha_{d}) \in \intO_{L}^{d}
                \mid
                \sigma(\alpha_{i}) = \alpha_{i} + \alpha_{i+1} (i < d),
                \sigma(\alpha_{d}) = \alpha_{d} \right\} \\
              &= \left\{ (\alpha, \delta(\alpha), \dots, \delta^{d-1}(\alpha)) \in \intO_{L}^{d}
                \mid
                \alpha \in \intO_{L}^{\delta^{d} = 0}\right\}.
\end{align}
\Cref{cor:basis-of-ker-delta} gives us a \(k[[t]]\)-basis of \(\intO_{L}^{\delta^{d}=0}\).
Then, we now have
\begin{lemma}\label{lem:v-function-valuation}
  With the notation above, we have
  \begin{align}
    \vect{v}(E^{*}) = \sum_{\substack{0 \leq i_{0} + pi_{1} + \dotsb + p^{n-1}i_{n-1} < d, \\
                                      0 \leq i_{0} , i_{1}, \dotsc, i_{n-1} < p}} %
    \ceil{- \frac{i_{0}v_{L}(\tilde{g}_{0}) + i_{1}v_{L}(\tilde{g}_{1}) +
    \dotsb + i_{n-1}v_{L}(\tilde{g}_{n-1})}
    {p^{n}}}.
  \end{align}
\end{lemma}
\begin{proof}
  Let \(n_{I}\) be an integer as in \cref{cor:basis-of-ker-delta}.  By
  \cref{prop:basis-ker-delta}, we find that the matrix
  \({\left(\delta^{m}(\tilde{\vect{g}}^{I} t^{n_{I}})\right)}_{I, m}\) is a triangular and
  that the diagonal components \(\delta^{a_{I}} (\tilde{\vect{g}}^{I} t^{n_{I}})\) are of the
  form \(ht^{n_{I}}\) (\(0\ne h \in k\)).  Then
  \begin{align}
    \vect{v}(E^{*})
    &= \frac{1}{\#G} v_{L}\left(\det{(\delta^{m}(\tilde{\vect{g}}^{I} t^{n_{I}}))}_{I, m}\right) \\
    &= \frac{1}{p^{n}} \sum_{0 \leq a_{I} < d} v_{L}(t^{n_{I}}) \\
    &= \sum_{0 \leq a_{I} < d} n_{I},
  \end{align}
  which is the desired conclusion.
\end{proof}

\subsection{Ramification jumps}\label{sec:ramification-jumps}
We next determine the values \(v_{L}(\tilde{g}_{l})\) by studying ramification of \(L/K\).  We
begin with recalling the notions of lower and upper ramification groups.  The basic reference
here is~\cite{Serre1979:Local}.  Let \(K\) be a complete discrete valuation field with the
perfect residue field of characteristic \(p > 0\).  Consider a finite Galois extension
\(L/K\).  We denote the valuation ring of \(L\) by \(\intO_{L}\) and the prime ideal of
\(\intO_{L}\) by \(\idlp_{L}\).  Put \(G \coloneq \Gal(L/K)\).

For each integer \(i \geq -1\), we set
\begin{equation}
  G_{i} \coloneq \{ \gamma \in G \mid
  \text{\(\gamma\) acts trivially on \(\intO_{L}/\idlp_{L}^{i+1}\)}\}
\end{equation}
and call it the \(i\)-th \emph{lower ramification group} of \(L/K\).
The lower ramification groups form a descending sequence \({\{G_{i}\}}_{i}\) of
normal subgroups of \(G\), and \(G_{i} = \{1\}\) for sufficiently large \(i\).

Let us next define upper ramification groups.  We put for \(t \in \setR_{\geq -1}\)
\begin{align}
  G_t &\coloneq G_{\ceil{t}}, \\
  (G_{0}: G_{t}) &\coloneq 1 \quad (t < 0).
\end{align}
We define the Hasse--Herbrand function
\(\varphi = \varphi_{L/K} \colon \setR_{\geq -1} \to \setR_{\geq -1}\) by
\begin{equation}
  \varphi_{L/K} (u) \coloneq \int_{0}^{u} \frac{dt}{(G_{0}: G_{t})}.
\end{equation}
This function \(\varphi_{L/K}\) is strictly increasing and a self-homeomorphism of
\(\setR_{\geq -1}\).  We denote \(\psi = \psi_{L/K} \coloneq \varphi_{L/K}^{-1}\).  For
\(u \geq -1\), we call \(G^{u} \coloneq G_{\psi_{L/K}(u)}\) the \(u\)-th \emph{upper
  ramification group} of \(L/K\).  As with the lower ramification groups, the upper
ramification groups form a descending sequence \(\{{G^{u}\}}_{u}\) of normal subgroups of
\(G\), and \(G^{u} = \{1\}\) for sufficiently large \(u\).

For each subgroup \(H \subset G\) of \(G\) and for each integer
\(i \geq -1\), we have
\begin{equation}\label{eq:prop:lower-rj}
  H_{i} = G_{i} \cap H.
\end{equation}
Here the filtration \({\{H_{i}\}}_{i}\) on \(H\) is induced from the \(H\)-extension derived
from \(L/K\).  Similarly, for each normal subgroup \(H \subset G\) of \(G\) and for each real
number \(u \geq -1\), we have
\begin{equation}\label{eq:prop:upper-rj}
  {(G/H)}^{u} = G^{u}H / H.
\end{equation}

We say that \(i\) is a \emph{lower ramification jump} of \(L/K\) if \(G_{i} \neq G_{i+1}\).
Also, we say that \(u\) is an \emph{upper ramification jump} of \(L/K\) if
\(G^{u} \neq G^{u+\epsilon}\) for all \(\epsilon > 0\).

We now restrict ourselves to the case of our principal interest where \(G = \setZ/p^{n}\setZ\)
and \(K = k((t))\).  From~\cite[p.~67,Corollary~2]{Serre1979:Local}, each graded piece
\(G_i/G_{i+1}\) is either \(1\) or \(\setZ/p\setZ\). Therefore, there are exactly \(n\) lower
ramification jumps and hence there are exactly \(n\) upper ramification jumps.

By definition, we can write
\begin{equation}
  \psi(u) = \int_{0}^{u} (G^{0}: G^{w}) dw.
\end{equation}
Let \(u_{0} < u_{1} < \dotsb < u_{n-1}\) be the upper ramification jumps.
Then, for any real number \(u\) with \(u_{i - 1} < u \leq u_{i}\), we have
\((G^{0}: G^{u}) = p^{i}\).  We remark that \(u_{0} \geq 0\) because the residue
field \(k\) is algebraically closed.  Therefore, we get
\begin{align}
  \psi(u_{i}) &= \int_{0}^{u_{i}} (G^{0}: G^{w}) dw \label{eq:Hasse--Herbrand-for-G} \\
              &= \int_{0}^{u_{0}} + \int_{u_{0}}^{u_{1}} +
                 \dotsb + \int_{u_{i-1}}^{u_{i}} (G^{0}: G^{w})dw \\
              &= u_{0} + (u_{1} - u_{0})p + \dotsb + (u_{i} - u_{i-1})p^{i}.
\end{align}
Note that \(\psi(u_{i})\) are the lower ramification jumps of \(L/K\) by definition.  In
particular, when \(G = \setZ/p\setZ\), the unique lower ramification jumps is equal to the
unique upper ramification jump; we call it simply the \emph{ramification jump}.

The following is immediate from the equality \(H_{i} = G_{i} \cap H\).

\begin{lemma}\label{lem:subgroups&lower-jumps}
  The highest lower ramification jump of \(K_{i}/K\) is equal to the ramification jump of
  \(K_{i}/K_{i-1}\).
\end{lemma}

\begin{lemma}\label{lem:quotients&upper-jumps}
  The highest upper jump of \(K_{m}/K\) is equal to the \((m + 1)\)-th upper jump \(u_m\) of
  \(L/K\).
\end{lemma}
\begin{proof}
  Note that the Galois group of \(K_{m}/K\) is the quotient of \(G = \setZ/p^{n}\setZ\) by the
  subgroup \(p^{n - m - 1}\setZ/p^{n}\setZ\).  From \cref{eq:prop:upper-rj}, the upper
  ramification jump of \(K_{m}/K\) are exactly \(u_{0}, u_{1}, \dotsc, u_{m}\), those jumps
  from a subgroup of \(G\) to another both of which contain \(p^{n - m - 1}\setZ/p^{n}\setZ\).
  The highest one among them is \(u_{m}\).
\end{proof}

The following claim follows from the lemmata above.

\begin{proposition}\label{prop:divide-jumps}
  The \((i+1)\)-th upper ramification jump of \(L/K\) is equal to the ramification jump of
  \(K_{i}/K_{i-1}\).
\end{proposition}
\begin{proof}
  The \((i + 1)\)-th upper ramification jump \(u_{i}\) is equal to the highest upper
  ramification jump of \(K_{i}/K\), which is equal to the ramification jump of
  \(K_{i}/K_{i - 1}\).
\end{proof}

\begin{lemma}\label{lem:rj-and-ord}
  The ramification jump of \(K_{i}/K_{i-1}\) is equal to \(v_{K_{i}}(\tilde{g}_{i})\).
\end{lemma}
\begin{proof}
  The proof given in~\cite[Proposition~2.10]{Yasuda2014:p-Cyclic} works for our situation,
  because of \(p \nmid v_{K_{i - 1}}(\tilde{f}_{i})\).
\end{proof}

From \cref{lem:v-function-valuation}, \(\bm{v}(\Spec L)\) is expressed in terms of valuations
of \(\tilde{g}_{i}\)'s, which are in turn related to upper ramification jumps of \(L/K\) by
the above results.  To determine \(\bm{v}(\Spec L)\), we now compute the upper ramification
jumps in terms of the corresponding representative Witt vectors.

\begin{theorem}
  Let \(L/K\) be a \(G\)-extension given by an equation \(\wp(\vect{g}) = \vect{f}\), where
  \(\vect{f}\) is reduced.  Then, the highest upper ramification jump is given by
  \begin{equation}
    \max\{- p^{n - 1 -i} v_{K}(f_{i}) \mid i = 0, 1, \dotsc, n - 1\}.
  \end{equation}
Here we follow the convention that \(v_{K}(0) = \infty\).
\end{theorem}
\begin{proof}
  For an integer \(m\), we define
  \begin{equation}
    W_n^{(m)}(K)\coloneq \{(f_0,\dots,f_{n-1})\mid p^{n-i-1}v_K(f_i) \geq m\}.
  \end{equation}
  From~\cite[p.~26, Corollary]{Brylinski1983:Theorie}, for
  \(\bm{f} \in W_{n}^{(-m)}(K) \setminus W_{n}^{(1-m)}(K)\), the corresponding extension
  \(L/K\) has Artin conductor \(m+1\) (for the character \(\chi \colon G \to \setC\) of any
  faithful irreducible \(G\)-representation over \(\setC\)).  From~\cite[Chapter~VI
  ,Proposition~5]{Serre1979:Local}, the highest upper ramification jump is \(m\).
\end{proof}

This theorem together with \cref{lem:quotients&upper-jumps} shows the following corollary:

\begin{corollary}\label{prop:upper-ramification-jumps}
  In the same situation as above, the upper ramification jumps of \(L/K\) are given by
  \begin{multline}\notag
    -v_{K}(f_{0}) \leq \max\{-pv_{K}(f_{0}), -v_{K}(f_{1})\} \leq \dotsb \\
    \dotsb \leq \max\{-p^{n - 1 - i} v_{K}(f_{i}) \mid i = 0, 1, \dotsc, n - 1\}.
  \end{multline}
\end{corollary}

We obtain the following as a conclusion of this section.

\begin{theorem}\label{thm:v-function}
  Let \(E^{*} = \Spec L\) be a \(G\)-cover of \(D^{*} = \Spec K\).  Assume that the
  \(G\)-extension \(L/K\) is defined by an equation \(\wp(\vect{g}) = \vect{f}\), where
  \(\vect{f} \in {\setRP_{k}}^{n}\) is a representative Witt vector of
  order \(- \vect{j}\) (\(\vect{j} \in {(\setN'\cup\{-\infty\})}^{n}\)) with
  \(j_{0} \ne - \infty\).  Put
  \begin{align}
    u_{i} &= \max \{p^{n - 1 - m} j_{m} \mid m = 0, 1, \dotsc, i - 1\}, \\
    l_{i} &= u_{0} + (u_{1} - u_{0})p + \dotsb + (u_{i} - u_{i - 1}) p^{i}.
  \end{align}
  Then
  \begin{equation}
    \vect{v}(E^{*}) = \sum_{\substack{0 \leq i_{0} + pi_{1} + \dotsb + p^{n-1}i_{n-1} < d, \\
                            0 \leq i_{0} , i_{1}, \dotsc, i_{n-1} < p}} %
    \ceil{\frac{i_{0}p^{n-1}l_{0} + i_{1}p^{n-2}l_{1} + \dotsb + i_{n-1}l_{n-1}}{p^{n}}}.
  \end{equation}
\end{theorem}
\begin{proof}
  By \cref{prop:upper-ramification-jumps}, \(u_{i}\) is the \((i+1)\)-th upper ramification
  jump of \(L/K\).  We can conclude that \(l_{i}\) is the \((i + 1)\)-th lower
  ramification jump of \(L/K\) by \cref{eq:Hasse--Herbrand-for-G}, and hence \(l_{i}\)
  is the lower ramification jump of \(K_{i}/K_{i-1}\) by \cref{prop:divide-jumps}.
  \cref{lem:rj-and-ord} shows that \(v_{L}(\tilde{g}_{i}) = -p^{n - 1 - i} l_{i}\).  By
  substituting them to the formula in \cref{lem:v-function-valuation}, we get the formula
  desired.
\end{proof}

\begin{remark}
  The previous theorem in particular shows that the function \(\bm{v}\) is constant on each
  \(\GCov(D; \vect{j})\).
\end{remark}

\section{Discrepancies of singularities}\label{sec:discrepancy}
In this section, we shall briefly review the wild McKay correspondence proved
in~\cite{Yasuda2019:Motivic} and explain how it relates discrepancies of quotient
singularities with the moduli space \(\GCov(D)\) and \(\bm{v}\)-function on it, following the
line of~\cite{Yasuda2014:p-Cyclic,Yasuda2019:Discrepancies}.

\subsection{Motivic integration}
The \emph{Grothendieck ring of varieties over \(k\)}, denoted by \(K_{0}\), is the abelian
group generated by isomorphic classes \([Y]\) of varieties over \(k\) subject to the following
relation: if \(Z\) is a closed subvariety of \(Y\), then \([Y] = [Y \setminus Z] + [Z]\).  It
has a ring structure by defining \([Y][Z] \coloneq [Y \times Z]\).  We denote
\(\mtvL \coloneq [\schA_{k}^{1}]\).  In application to the McKay correspondence, we need also
the following relation.  For a morphism \(f \colon Y \to Z\) of \(k\)-varieties and an integer
\(m \geq 0\), if every geometric fiber of \(f\) say over an algebraic closed field \(L\) is
universally homeomorphic to the quotient of \(\schA_{L}^{m}\) by some linear finite group
action, then \([Y] = \mtvL^{m}[Z]\).  We define \(K_{0}'\) to be the quotient of \(K_{0}\) by
this relation.  We denote by \(\rngM' \coloneq K_{0}'[\mtvL^{-1}]\) the localization by
\(\mtvL\).  Subgroups \(F_{m} \coloneq \langle [X] \mtvL^{i} \mid \dim X + i \leq -m \rangle\)
of \(\rngM'\) form a filtration.  We define \(\rngMhat' \coloneq \varprojlim \rngM'/F_{m}\),
which is again a commutative ring and complete with respect to the induced topology.

For \(n \in \setN\), let \(\pi_{n} \colon \spcJ X \to \spcJ[n] X\) be the truncation map to
\(n\)-jets.  We call a subset \(C \subset \spcJ X\) \emph{stable} if there exists
\(n \in \setN\) such that \(\pi_{n}(C)\) is a constructible subset of \(\spcJ[n] X\),
\(C = \pi_{n}^{-1}(\pi_{n}(C))\) and the map \(\pi_{m+1}(C) \to \pi_{m}(C)\) is a piecewise
trivial \(\schA_{k}^{d}\)-bundle for every \(m \geq n\).  We define the measure \(\mu_{X}(C)\)
of a stable subset \(C \subset \spcJ X\) by
\begin{equation}
  \mu_{X}(C) \coloneq [\pi_{n}(C)] \mtvL^{-nd} \in \rngMhat'.
\end{equation}
For a more general measurable subset of \(\spcJ X\), we define its measure as the limit of
ones of stable subsets.

Let \(C \subset \spcJ X\) be a measurable subset and \(F \colon C \to \setZ \cup \{\infty\}\)
a function on it.  We say that \(F\) is \emph{measurable} if every fiber of \(F\) is measurable.
Now we define the integral
\begin{equation}
  \int_{C} \mtvL^{F} \coloneq \sum_{m \in \setZ} \mu_{X}(F^{-1}(m)) \mtvL^{m}
    \in \rngMhat \cup \{\infty\}.
\end{equation}
Note that \(\int_{C} \mtvL^{F}\) does not necessarily converge.

\subsection{Stringy motives}
To state the wild McKay correspondence theorem, we shall define the stringy motive.  Firstly,
we shall recall basic notations concerning singularities.  Let \(X\) be a normal
\(k\)-variety, \(f \colon Y \to X\) a modification (proper birational morphism) such that
\(Y\) is a normal \(k\)-variety.  Assume that both the exceptional locus \(\Exc(f)\) and the
preimage \(f^{-1}(X_{\txtsing})\) of \(X_{\txtsing}\) are pure-dimension \(d - 1\).  We call
such a morphism an \emph{admissible modification}.  Note that the last condition implies
\(f^{-1}(X_{\txtsing}) \subset \Exc(f)\).  When \(X\) is \(\setQ\)-Gorenstein ,we can define
the \emph{relative canonical divisor} \(K_{f}\) in the usual way, which is a \(\setQ\)-Weil
divisor with a support contained in \(\Exc(f)\).  Let
\(\Exc(f) = \bigcup_{i \in \mathcal{E}_{f}} E_{i}\) and
\(f^{-1}(X_{\txtsing}) = \bigcup_{i \in \mathcal{S}_{f}} E_{i}\) be the decomposition into
irreducible components with \(\mathcal{S}_{f} \subset \mathcal{E}_{f}\) and write
\(K_{f} = \sum_{i} a_{i} E_{i}\).  We call \(a_{i}\) the \emph{discrepancy} of \(E_{i}\) with
respect to \(X\) and define
\begin{equation}
  d(X)  = \discrep(\text{center} \subset X_{\txtsing}; X)
  \coloneq \inf_{f} \min_{i \in \mathcal{S}_{f}} a_{i}.
\end{equation}
Here \(f\) runs over admissible modifications of \(X\).  We say that \(X\) is \emph{terminal}
(resp.\ \emph{canonical}, \emph{log terminal}, \emph{log canonical}) if \(d(X) > 0\) (resp.\
\(\geq 0\), \(> -1\), \(\geq -1 \)).  Note that if \(d(X) < -1\), then \(d(X) = - \infty\).

We also need to consider log pairs as is usual in birational geometry.  By a log pair, we mean
the pair \((X, \Delta)\) of a normal \(\setQ\)-Gorenstein variety \(X\) and a
\(\setQ\)-Cartier \(\setQ\)-Weil divisor \(\Delta\) on it.  We say that a log pair
\((X, \Delta)\) is \emph{klt} (resp.\ \emph{lc}) if for any admissible modification
\(f \colon Y \to X\), \(K_{f} - f^{*}\Delta\) has coefficients \(> 1\) (resp.\ \(\geq 1\)).

\begin{remark}
  If \(\Delta = 0\), then the notions klt (Kawamata log terminal), plt (pure log terminal) and
  dlt (divisorial log terminal) coincide.  In this case, we say that \(X\) is log terminal.
  Furthermore, if \(K_{X}\) is Cartier, then log terminal implies canonical.  For details,
  see~\cite[pp.~42--43]{Kollar2013:Singularities}.
\end{remark}

\begin{definition}
  Let \(X\) be a normal variety of pure-dimension \(d\).  We assume that the canonical sheaf
  \(\omega_{X}\) is invertible.  We define \emph{\(\omega\)-Jacobian ideal}
  \(\mathcal{J}_{X}\) by
  \begin{equation}
    \mathcal{J}_{X} \omega_{X} = \Img \left( \bigwedge^{d} \Omega_{X/k} \to \omega_{X}\right).
  \end{equation}
  For a log pair \((X, \Delta)\), we define the \emph{stringy motive} \(\Mst(X, \Delta)\) by
  \begin{equation}
    \Mst(X, \Delta) \coloneq \int_{\spcJ X} \mtvL^{\ord \Delta + \ord \mathcal{J}_{X}}.
  \end{equation}
  Here \(\ord\) denotes the order function associated to a divisor or an ideal sheaf.
\end{definition}

\begin{remark}
  The invariant ring by a linear action of a \(p\)-group is a UFD
  (\cite[Theorem~3.8.1]{Campbell2011:Modular}).  Especially, in our situation where
  \(G = \setZ/p^{n}\setZ\) linearly acts on \(\schA_{k}^{d}\), the quotient variety
  \(X \coloneq \schA_{k}^{d}/G\) is \(1\)-Gorenstein, that is, \(\omega_{X}\) is invertible.
\end{remark}

\subsection{The wild McKay correspondence}
Let \(G\) be a finite group (not necessarily cyclic of prime power order).  Generally, we can
construct the moduli space \(\GCov(D)\) of \(G\)-covers of \(D = \Spec k[[t]]\).  Furthermore,
we can define a measure on it.  For a locally constructible function
\(F \colon \GCov(D) \to \setQ\), we can define the integral \(\int_{\GCov(D)} \mtvL^{F}\).  We
remark that the \(\bm{v}\)-function is locally constructible.  For details,
see~\cite{Tonini2019:Moduli}.

\begin{theorem}[{\cite[Corollary~16.3]{Yasuda2019:Motivic}}]\label{Yasuda2019:Cor14.4}
  Let \(G\) be a finite group.  Assume that \(G\) acts on \(\schA_{k}^{d}\) linearly and
  effectively.  Put \(X \coloneq \schA_{k}^{d} / G\) and let \(\Delta\) be the \(\setQ\)-Weil
  divisor on \(X\) such that \(\schA_{k}^{d} \to (X, \Delta)\) is crepant.  Then we have
  \begin{equation}
    \Mst(X, \Delta) = \int_{\GCov(D)} \mtvL^{d - \bm{v}}.
  \end{equation}
\end{theorem}

We shall consider the case \(G = \setZ/p^{n}\setZ\), which is of our principal interest.  In
this case, we can describe the measure on \(\GCov(D)\) explicitly as follows.  For a
constructible subset \(C\) of \(\GCov(D; \vect{j})\), then the measure \(\nu(C)\) is given by
\begin{equation}
  \nu(C) \coloneq [C] \in \rngMhat'.
\end{equation}
Suppose that \(F \colon \GCov(D) \to \setQ\) is constant on each stratum
\(\GCov(D; \vect{j})\).  We write \(F(\vect{j}) = F(\GCov(D; \vect{j}))\).  Then we can write
\begin{align}
  \int_{\GCov(D)} \mtvL^{F} &= \sum_{r \in \setQ} \nu(F^{-1}(r)) \mtvL^{r} \\
  &= \sum_{\vect{j}} \nu(\GCov(D; \vect{j})) \mtvL^{F(\vect{j})}.
\end{align}
Putting \(\vect{j} = (j_{0}, j_{1}, \dotsc, j_{n-1}) \in {(\setN' \cup \{- \infty\})}^{n}\),
we have
\begin{equation}
  \nu(\GCov(D; \vect{j})) =
    \prod_{j_{l} \neq - \infty} (\mtvL - 1) \mtvL^{j_{l} - 1 - \floor{j_{l}/p}},
\end{equation}
and hence
\begin{equation}
  \int_{\GCov(D)} \mtvL^{F} = \sum_{\vect{j}} \left(
    \prod_{j_{l} \neq - \infty} (\mtvL - 1) \mtvL^{j_{l} - 1 - \floor{j_{l}/p}}
  \right) \mtvL^{F(\vect{j})}.
\end{equation}

In \cref{Yasuda2019:Cor14.4}, if the \(G\)-action has no pseudo-reflection, then
\(\Delta = 0\).  If the action is indecomposable, we can check easily whether \(G\) has
pseudo-reflections or not as follows.

\begin{lemma}\label{lem:Jordan-form-of-p-power}
  Let \(J\) be a Jordan block of size \(d\) with eigenvalue \(1\) over \(k\).  We write
  \(d = qp + r\) (\(0 \leq r < p\)).  Then the Jordan standard form of \(J^{p}\) has \(r\)
  Jordan blocks of size \(q + 1\) and \(p-r\) Jordan blocks of size \(q\); in particular, it
  has exactly \(p\) blocks.
\end{lemma}
\begin{proof}
  In general, the following holds: Let \(A\) be a square matrix of size \(d\).  We denote by
  \(C_{m}(\lambda)\) the number of Jordan block of size \(m\) with eigenvalue \(\lambda\) in
  the Jordan standard form of \(A\).  Then, we have
  \begin{equation}
    C_{m}(\lambda) = \rank {(A - \lambda E)}^{m - 1} - 2 \rank {(A - \lambda E)}^{m}
      + \rank {(A - \lambda E)}^{m + 1},
  \end{equation}
  where \(E\) denotes the identity matrix.  To prove the formula, we may assume that \(A\) is
  a Jordan block say with eigenvalue \(\lambda'\).  If \(\lambda' \neq \lambda\), the formula
  is obvious.  Let us assume \(\lambda' = \lambda\).  Then \(A - \lambda E\) is nilpotent and
  \(\rank {(A - \lambda E)}^{m} = \max \{0, d - m\}\).  By direct computation, we get the
  formula.

  Set \(A = J^{p}\).  It is easy to see that
  \(\rank {(J^{p} - E)}^{m} =\rank {(J - E)}^{pm} = \max \{0, d - pm\}\).  Especially, we have
  \(\rank {(J^{p} - E)}^{q-1} = d - p(q-1) = p+ r\), \(\rank {(J^{p} - E)}^{q} = d - pq = r\)
  and \(\rank {(J^{p} - E)}^{q + 1} = \rank {(J^{p} - E)}^{q + 2} = 0\).  Therefore, we get
  \(C_{q + 1}(1) = r\) and \(C_{q}(1) = p - r\).  The equality
  \(r(q + 1) + (p - r)q = qp + r = d\) completes the proof.
\end{proof}

\begin{lemma}\label{lem:has-p.refl-or-not}
  Let \(J_{d}\) be the Jordan block of size \(d\) with eigenvalue \(1\)
  (\(1 \leq d \leq p^{n}\)).  For \(1 \leq m < n\), \(J_{d}^{p^{m}}\) is a pseudo-reflection
  if and only if \(d = p^{m} + 1\).
\end{lemma}
\begin{proof}
  We shall prove by induction on \(m\).  When \(m = 1\), the claim follows immediately from
  \cref{lem:Jordan-form-of-p-power}.  Let \(m > 1\).  We write \(d = qp + r\)
  (\(0 \leq r < p\)).  Then
  \begin{align}
    J_{d}^{p^{m}}
    &= {(J_{d}^{p})}^{p^{m - 1}} \\
    &\equiv {\left( J_{q + 1}^{\oplus r} \oplus
      J_{q}^{\oplus p - r}\right)}^{p^{m - 1}} \\
    &= {\left(J_{q + 1}^{p^{m - 1}}\right)}^{\oplus r} \oplus
      {\left(J_{q}^{p^{m - 1}}\right)}^{\oplus p - r},
  \end{align}
  where \(\equiv\) denotes the similarity equivalence.  The matrix \(J_{d}^{p^{m}}\) is a
  pseudo-reflection if and only if one of the following holds:
  \begin{enumerate}
  \item \(J_{q + 1}^{p^{m - 1}}\) is a pseudo-reflection, \(r = 1\) and
    \(J_{q}^{p^{m - 1}} = 1\),
  \item \(J_{q}^{p^{m - 1}}\) is a pseudo-reflection, \(p - r = 1\) and
    \(J_{q + 1}^{p^{m - 1}} = 1\).
  \end{enumerate}
  In the latter case (2), by the induction hypothesis, we get \(q = p^{m - 1} + 1\), which
  contradicts the equality \(J_{q + 1}^{p^{m - 1}} = 1\).  In the former case (1), we get
  \(q + 1 = p^{m - 1} + 1\) and hence \(d = p^{m} + 1\).  Conversely, it is obvious that
  \(J_{p^{m} + 1}^{p^{m}}\) is a pseudo-reflection.
\end{proof}

\begin{corollary}\label{cor:pseudo-ref}
  Let \(J\) be a matrix of the Jordan normal form with a unique eigenvalue \(1\).
  \begin{enumerate}
  \item For a given integer \(m \geq 0\), the matrix \(J^{p^m}\) is a pseudo-reflection if and
    only if \(J\) has one Jordan block of size \(p^{m}+1\) and all the other blocks have size
    \(\le p^{m}\).
  \item Let \(p^n\) be the order of \(J\). The group
    \(\langle J \rangle \cong \setZ/p^{n}\setZ\) contains a pseudo-refection if and only if
    \(J\) has one Jordan block of size \(p^{n - 1} + 1\) and all the other blocks have size
    \(\le p^{n - 1}\). Moreover, if this is the case, the pseudo-reflections in the group are
    \(J^{i p^{n - 1}}\), \(1 \leq i \leq p - 1\).
  \end{enumerate}
\end{corollary}

\begin{proof}
  (1). The ``if'' part immediately follows from the last lemma. If there are at least two
  blocks say \(A\) and \(B\) of size \(> p^{m}\), then neither \(A^{p^{m}}\) or \(B^{p^{m}}\)
  are the identity matrix. This shows that \(J^{p^m}\) is not a pseudo-reflection.  If there
  is no block of size \(> p^{m}\), then \(J^{p^{m}} = 1\), which is not a
  pseudo-reflection. Thus, for \(J^{p^{m}}\) being a pseudo-reflection, \(J\) needs to have
  one and only one block of size \(> p^{m}\) whose \(p^{m}\)-th power is a
  pseudo-reflection. Again from the last lemma, this block needs to have size \(p^{m}+1\).

  (2). For the group having pseudo-reflections, the matrix \(J\) needs to be of the form as in
  (1) for some \(m\). Because of the order, we have \(m = n - 1\). Conversely, if \(J\) is of
  this form for \(m = n - 1\), then the group contains the pseudo-reflection
  \(J^{p^{n - 1}}\). Thus the first assertion of (2) holds. To show the second assertion, we
  first note that when two elements \(A\), \(B\) of \(\langle J \rangle\) generate the same
  subgroup, then \(A\) is a pseudo-reflection if and only if \(B\) is a
  pseudo-reflection. Therefore we only need to consider the \(p\)-powers \(J^{p^{m}}\). The
  only pseudo-reflection among them is the one for \(m= n - 1\). The second assertion follows.
\end{proof}

\begin{proposition}\label{prop:boundary}
  Suppose that \(G=\setZ/p^{n}\setZ\) acts on \(\schA_{k}^{d}\) linearly and effectively and
  that there exists a pseudo-reflection. Let \(H \subset \schA_{k}^{d}\) be the hyperplane
  fixed by a pseudo-reflection in \(G\) (this hyperplane is independent of the
  pseudo-reflection from the above corollary).  Let \(\overline{H}\) be the image of \(H\) in
  the quotient variety \(\schA_{k}^{d}/G\) with the reduced structure. Then the map
  \(\schA_{k}^{d} \to (\schA_{k}^{d}/G, (p - 1) \overline{H})\) is crepant.
\end{proposition}

\begin{proof}
  Let \(V \coloneq \schA_{k}^{d}\) and \(X \coloneq V/G\). Let \(v \in H\) be a general
  \(k\)-point whose stabilizer subgroup \(S \subset G\) has order \(p\) and let
  \(x \in \overline{H}\) be its image. To compute the right coefficient of the boundary
  divisor on \(X\), it is enough to consider the morphism
  \(\Spec \hat{\shfO}_{V, v} \to \Spec \hat{\shfO}_{X, x}\) between the formal neighborhoods
  of \(v\) and \(x\). This morphism is isomorphic to the one similarly defined for the
  quotient morphism \(V \to V/S\) associated to the induced action of \(S =\setZ/p\setZ\) on
  \(V\) with pseudo-reflections. In this case, we know from~\cite{Yasuda2014:p-Cyclic} that
  the coefficient of the boundary divisor is \(p - 1\). This shows the proposition.
\end{proof}

\subsection{Discrepancies of singularities}
We shall recall the relation between stringy motives and discrepancies.

\begin{proposition}[{\cite[Proposition~6.6]{Yasuda2014:p-Cyclic}}]\label{Yasuda2014:Prop6.6}
  If the stringy motive \(\Mst(X, \Delta)\) converges, then the pair \((X, \Delta)\) is klt.
  Furthermore, if there exists a resolution \(f \colon Y \to X\) such that
  \(K_{Y} - f^{*}(K_{X} + \Delta)\) is a simple normal crossing \(\setQ\)-Cartier divisor,
  then the converse is also true.
\end{proposition}

For a measurable subset \(U\) of \(\spcJ X\), we define
\begin{equation}
  \lambda(U) \coloneq \dim \left(\int_{U} \mtvL^{\ord \mathcal{J}_{X}} d \mu_{X}\right),
\end{equation}
provided that the integration converges.  We say that a measurable subset \(U\) of \(\spcJ X\)
is \emph{small} if the relevant integration converges.  We also denote the truncation map by
\(\pi \colon \spcJ X \to \spcJ[0] X = X\).

The following proposition tell us that we can estimate the discrepancies of
\(\schA_{k}^{d}/G\) by computing the integration \(\int_{\GCov(D)} \mtvL^{d - \bm{v}}\).

\begin{proposition}[{\cite[Proposition~2.1]{Yasuda2019:Discrepancies}}]%
  \label{prop:Yasuda2019Prop2.1}
  Let \(C_{r} \subset \pi_{X}^{-1}(X_{\txtsing})\), (\(r \in \setN\)) be a countable
  collection of small measurable subset such that \(\pi^{-1}(X_{\txtsing})\) and
  \(\bigcup_{r \in \setN} C_{r}\) coincide outside a measurable subset.  Then
  \begin{equation}
    d(X) = d - 1 - \sup_{r} \lambda(C_{r}).
  \end{equation}
\end{proposition}

We now suppose that \(X \coloneq \schA_{k}^{d} / G\) the quotient variety associated to an
effective linear action of \(G = \setZ/p^{n}\setZ\).  The quotient morphism
\(\schA_{k}^{d} \to X\) and an arc \(D = \Spec K \to X\) induce a \(G\)-cover of \(D\), unless
the arc maps into the branch locus of \(\schA_{k}^{d} \to X\); the last exceptional case
occurs only for arcs in a measure zero subset of \(\spcJ X\).  For
\(\vect{j} \in {(\setN' \cup \{- \infty\})}^{n}\), let
\(M_{\vect{j}} \subset \pi^{-1}(X_{\txtsing})\) be the locus of arcs including a \(G\)-cover
\(E\) with \(\ord E = - \vect{j}\).  The collection of \(M_{\vect{j}}\) satisfies the
condition of \cref{prop:Yasuda2019Prop2.1}.  Suppose that \(G\) has no pseudo-reflection.  As
a variant of \cref{Yasuda2019:Cor14.4}, for
\(\vect{j} \neq (- \infty, -  \infty, \dotsc, - \infty)\), we have
\begin{equation}\label{eq:integral-over-Mj}
  \int_{M_{\vect{j}}} \mtvL^{\ord \mathcal{J}_{X}}
  = \int_{\GCov(D; \vect{j})} \mtvL^{d - \bm{v}}
  = [\GCov(D; \vect{j})] \mtvL^{d - \bm{v}(\vect{j})}
\end{equation}
and
\begin{equation}
  \lambda(M_{\vect{j}}) = \dim \nu(\GCov(D; \vect{j})) + d - \bm{v}(\vect{j}).
\end{equation}
The case \(\vect{j} = (- \infty, - \infty, \dotsc, - \infty)\) corresponds to the trivial
\(G\)-cover \(\coprod D \to D\).  We have
\begin{equation}
  \int_{M_{(- \infty, - \infty, \dotsc, - \infty)}} \mtvL^{\ord \mathcal{J}_{X}}
  = [R/G] = [B],
\end{equation}
where \(R \subset \schA_{k}^{d}\) and \(B \subset X\) are the ramification and the branch loci
of \(\schA_{k}^{d} \to X\) respectively.  In particular,
\begin{equation}
  \lambda(M_{(- \infty, - \infty, \dotsc, - \infty)}) = \dim R = \dim B.
\end{equation}
These formulae for \(\lambda\) together with \cref{prop:Yasuda2019Prop2.1} enable us to
estimate \(d(X)\) in terms of the \(\bm{v}\)-function in theory.  We shall carry it out in the
case \(n = 2\);  computation in this case is already rather complicated.

\section{The case \texorpdfstring{\(G=\setZ/p^2\setZ\)}{G=Z/p2Z}}\label{sec:case:n=2}
As an application of \cref{thm:v-function}, we shall compute \(\bm{v}\)-function for the case
\(G = \setZ/p^{2}\setZ\) and give a criterion for convergence the stringy motive
\(\Mst(\schA_{k}^{d}/G, \Delta)\).

\subsection{Some invariants}\label{sec:some-invariants}
For integers \(d\), \(j_0\) \(j_1\) with \(0<d\leq p^2\), writing \(d = qp + r\)
(\(0 \leq r < p\)), we define
\begin{align}
  \bm{e}_{d}^{>}(j_{0}, j_{1}) = \bm{e}_{d}^{>}(j_0)
    &\coloneq
      \sum_{\substack{0 \leq i_{0}, i_{1} < p, \\ 0 \leq i_{0} + i_{1} p < d}}
      \ceil{\frac{p i_{0} j_{0} + (p^{2} - p + 1) i_{1}j_{0}}{p^{2}}}, \\
  \bm{e}_{d}^{<}(j_{0}, j_{1})
    &\coloneq
      \sum_{\substack{0 \leq i_{0}, i_{1} < p, \\ 0 \leq i_{0} + i_{1} p < d}}
      \ceil{\frac{p i_{0} j_{0} + (-(p - 1) j_{0} + p j_{1}) i_{1}}{p^{2}}}.
\end{align}

Let \(V\) be a \(d\)-dimensional \(G\)-representation.  We do not assume \(0 < d \leq p^{n}\)
here.  It uniquely decomposes into indecomposables as
\begin{equation}
  V = \bigoplus_{\alpha = 1}^{a} V_{d_{\alpha}} \quad
  \left(1 \leq d_{\alpha} \leq p^{2}, \sum_{\alpha = 1}^{a} d_{\alpha} = d\right),
\end{equation}
where \(V_{d_{\alpha}}\) denotes the unique indecomposable \(G\)-representation of dimension
\(d_{\alpha}\). Components \(V_{d_{\alpha}}\) correspond to Jordan blocks in the Jordan normal
form of a generator of \(G\).  For integers \(j_0, j_1\), we define
\begin{align}
  \bm{e}_{V}^{>}(j_{0})
  &\coloneq \sum_{\alpha = 1}^{a} \bm{e}_{d_{\alpha}}^{>}(j_{0}), &
  \bm{e}_{V}^{<}(j_{0}, j_{1})
  &\coloneq \sum_{\alpha = 1}^{a} \bm{e}_{d_{\alpha}}^{<}(j_{0}, j_{1}).
\end{align}
For a connected \(G\)-cover \(E^*\) of \(D^*\) with order \(- \vect{j}= (- j_0, - j_1)\), from
\cref{thm:v-function}, we have
\begin{equation}
  \bm{v}_{V}(E^*) =
  \begin{cases}
     \bm{e}_V^> (j_0)      & \text{if \(pj_0 > j_1\)}, \\
     \bm{e}_V^< (j_0, j_1) & \text{if \(pj_0 < j_1\)}.
  \end{cases}
\end{equation}
Actually the functions \(\bm{e}_{V}^{>}\) and \(\bm{e}_{V}^{<}\) are both the sum of a linear
function and a periodic function.  To describe the linear part, we introduce some
invariants. For \(0 < d \leq p^{2}\), again writing \(d = qp + r\) (\(0 \leq r < p\)), we
define
\begin{align}
  A_{d}
    &\coloneq
      \sum_{\substack{0 \leq i_{0}, i_{1} < p, \\ 0 \leq i_{0} + i_{1} p < d}} i_{0}
      = \frac{qp(p - 1)}{2} + \frac{r(r - 1)}{2}, \\
  B_{d}
    &\coloneq
      \sum_{\substack{0 \leq i_{0}, i_{1} < p, \\ 0 \leq i_{0} + i_{1} p < d}} i_{1}
      = \frac{qp(q - 1)}{2} + qr.
\end{align}
For a \(G\)-representation \(V\) with decomposition as above, we define
\begin{equation}
  A_{V}
  \coloneq \sum_{\alpha = 1}^{a} A_{d_{\alpha}}, \,
   B_{V}
  \coloneq \sum_{\alpha = 1}^{a} B_{d_{\alpha}}
\end{equation}
and
\begin{equation}
  C_{V}^{>}    \coloneq p A_{V} + (p^{2} - p + 1) B_{V}, \quad
  C_{V}^{<, 0} \coloneq p A_{V} - (p - 1) B_{V}, \quad
  C_{V}^{<, 1} \coloneq p B_{V}.
\end{equation}
Note that all these invariants are integers.

\begin{lemma}\label{lem:reminder1}
For integers \(n_{i}\), \(s_{i}\) (\(i = 0, 1\)), we have
  \begin{align}
    \bm{e}_{V}^{>} (n_{0} p^{2} + s_{0} )
    &= C_V^{>} \cdot n_{0} + \bm{e}_{V}^{>} (s_{0}), \\
    \bm{e}_{V}^{<}(n_{0}p^{2} + s_{0}, n_{1}p^{2} + s_{1})
    &=
    C_{V}^{<, 0} n_{0} + C_{V}^{<, 1} n_{1} + \bm{e}_{V}^{<}(s_{0}, s_{1}).
  \end{align}
\end{lemma}
\begin{proof}
  Without loss of generality, we may assume that \(V = V_{d}\).  By direct computation, we get
  \begin{align}
    &
      \bm{e}_{d}^{>}(n_{0} p^{2} + s_{0}) \\
    &= \sum_{\substack{0 \leq i_{0}, i_{1} < p, \\ 0 \leq i_{0} + i_{1} p < d}}
    \ceil{\frac{%
      p i_{0} (n_{0}p^{2} + s_{0}) + (p^{2} - p + 1) i_{1} (n_{0}p^{2} + s_{0})%
    }{p^{2}}} \\
    &= \sum_{\substack{0 \leq i_{0}, i_{1} < p, \\ 0 \leq i_{0} + i_{1} p < d}}
      \ceil{p i_{0} n_{0} + (p^{2} - p + 1) i_{1} n_{0}
      + \frac{p i_{0} s_{0} + (p^{2} - p + 1) i_{1} s_{0}}{p^{2}}} \\
    &= \left(
      p \sum_{\substack{0 \leq i_{0}, i_{1} < p, \\ 0 \leq i_{0} + i_{1} p < d}} i_{0}
      + (p^{2} - p + 1)
        \sum_{\substack{0 \leq i_{0}, i_{1} < p, \\ 0 \leq i_{0} + i_{1} p < d}} i_{1}
      \right) n_{0} + \bm{e}_{d}^{>}(s_{0})\\
    &=
      (p A_{d}
      + (p^{2} - p + 1)B_d) n_{0} + \bm{e}_{d}^{>}(s_{0}) \\
    &= C_{V_{d}}^{>} \cdot n_{0} + \bm{e}_{d}^{>}(s_{0}),
  \end{align}
  which induces the first equality.  Similarly, we get
  \begin{align}
    &
      \bm{e}_{d}^{<}(n_{0} p^{2} + s_{0}, n_{1} p^{2} + s_{1}) \\
    &= \sum_{\substack{0 \leq i_{0}, i_{1} < p, \\ 0 \leq i_{0} + i_{1} p < d}}
      \ceil{
        \frac{
          p i_{0} (n_{0}p^{2} + s_{0}) + (-(p - 1)(n_{0}p^{2} + p (n_{1}p^{2} + s_{1}))) i_{1}
        }{p^{2}}} \\
    &= \sum_{\substack{0 \leq i_{0}, i_{1} < p, \\ 0 \leq i_{0} + i_{1} p < d}}
      \ceil{
        \left(p i_{0} - (p - 1) i_{1}\right) n_{0}
        + p i_{1} n_{1}
        + \frac{
            p i_{0} s_{0} + (-(p - 1) s_{0} + p s_{1}) i_{1}
          }{p^{2}}} \\
    &= (p A_{d} - (p - 1) B_{d}) n_{0} + p B_{d} + \bm{e}_{d}^{<}(s_{0}, s_{1}) \\
    &=
      C_{V_{d}}^{<, 0} \cdot n_{0} + C_{V_{d}}^{<, 1} \cdot n_{1}
      + \bm{e}_{d}^{<}(s_{0}, s_{1}),
  \end{align}
  which completes the proof.
\end{proof}

We will need the following upper and lower bounds of \(A_{d}\) and \(B_{d}\) later.

\begin{lemma}\label{lem:bounds-for-A&B}
  With notations as above, we have
  \begin{align}
    \frac{d (p - 1)}{2} - \frac{p^{2}}{8} &\leq A_{d} \leq \frac{d (p - 1)}{2}, \\
    \frac{d (d - p)}{2p} &\leq B_{d} \leq \frac{d (d - p)}{2p} + \frac{p}{8}.
  \end{align}
\end{lemma}
\begin{proof}
  By definition, we have
  \begin{align}
    A_{d} &= \frac{(d - r)(p - 1)}{2} + \frac{r (r - 1)}{2} \\
          &= \frac{d (p - 1)}{2} + \frac{r (r - p)}{2}.
  \end{align}
  From the inequality of arithmetic and geometric means,
  \begin{align}
    0 & \geq r(r-p)/2               \\
      & \geq - {(r + (r - p))}^{2} /8 \\
      & = -p^{2} / 8.
  \end{align}
  Therefore, we get
  \begin{equation}
    \frac{d (p - 1)}{2} - \frac{p^{2}}{8} \leq A_{d} \leq \frac{d (p - 1)}{2}.
  \end{equation}

  Similarly, we have
  \begin{align}
    B_{d} &= \frac{(d - r) \left(\frac{d - r}{p} - 1\right)}{2} + \frac{d - r}{p} r \\
          &= \frac{d (d - p)}{2p} - \frac{r (r - p)}{2p},
  \end{align}
  which completes the proof.
\end{proof}

\subsection{A criterion for convergence}
With the notation above, we state the main result of this section as follows.

\begin{theorem}\label{thm:cond-converge}
  The integral \(\int_{\GCov(D)} \mtvL^{d - \bm{v}}\) converges if and only if the
  following inequalities hold:
  \begin{align}
    B_{V} &\geq p, \\
    C_{V}^{>} &\geq p^{3} - p + 1.
  \end{align}
\end{theorem}

\begin{corollary}\label{prop:cond-klt}
  Let \(X \coloneq V/G\) be the quotient and \(\Delta\) the \(\setQ\)-Weil divisor on \(X\)
  such that \(V \to (X, \Delta)\) is crepant.  If the inequalities \(B_{V} \geq p\) and
  \(C_{V}^{>} \geq p^{3} - p + 1\) holds, then the pair \((X, \Delta)\) is klt.  Furthermore,
  if there exists a log resolution of \((X, \Delta)\), then the converse is also true.
\end{corollary}
\begin{proof}
  If the inequalities holds, then the integral \(\int_{\GCov(D)} \mtvL^{d - \bm{v}}\) converges
  from \cref{thm:cond-converge} and hence the stringy motive \(\Mst(X, \Delta)\) also
  converges from \cref{Yasuda2019:Cor14.4}.  We can prove the claim
  from \cref{Yasuda2014:Prop6.6}.
\end{proof}

The rest of this subsection is devoted to the proof of \cref{thm:cond-converge}.

According to the decomposition of \({(\setN' \cup \{-\infty\})}^2\) into four parts
\begin{equation}
  \{(-\infty,-\infty)\}, \, \{(-\infty,j)\mid j\in \setN'\}, \,
 \{(j_0,j_1)\mid pj_0 >j_1 \} , \,
 \{(j_0,j_1)\mid j_0\ne -\infty ,\, pj_0 < j_1 \},
\end{equation}
we divide the integral over \(\GCov(D)\) into four parts:
\begin{align}
  \int_{\GCov(D)} \mtvL^{d - \bm{v}} =
  & \mtvL^{d} \\
\label{1st-sum}  &+ \sum_{j}
    \nu(\GCov(D; -\infty, j)) \mtvL^{d - \bm{v}(-\infty, j)} \\
\label{2nd-sum}  &+ \sum_{ pj_{0} > j_{1}}
    \nu(\GCov(D; j_{0}, j_{1})) \mtvL^{d - \bm{v}(j_{0}, j_{1})} \\
\label{3rd-sum}  &+ \sum_{ pj_{0} < j_{1}}
    \nu(\GCov(D; j_{0}, j_{1})) \mtvL^{d - \bm{v}(j_{0}, j_{1})}.
\end{align}
The integral converges if and only if all the three sums on the right hand side converge.  We
will study convergence of these sums in turn.

\subsubsection{Sum \cref{1st-sum}}\label{sec:1st-sum}
This part corresponds to the \(G\)-covers of \(D\) which have \(p\) connected components.
Each component is then an \(H\)-cover with \(H \subset G\) the subgroup of order \(p\).  Let
\(E\) be such a \(G\)-cover and let \(E'\) be a connected component of it.  We have
\begin{equation}
  \bm{v}(E)= \bm{v}'(E'),
\end{equation}
where \(\bm{v}'\) denotes the \(\bm{v}\)-function \(\bm{v}' \colon \GCov[H](D) \to \setQ\)
associated to the induced \(H\)-action on \(\schA_k^d\).

By \cref{lem:Jordan-form-of-p-power}, we find that the restriction of the indecomposable
\(G\)-representation \(V_{d_{\alpha}}\) to \(H\) is isomorphic to
\begin{equation}\label{eq:induced-subrepresentation}
  W_{q_{\alpha} + 1}^{\oplus r_{\alpha}} \oplus W_{q_{\alpha}}^{\oplus p - r_{\alpha}},
\end{equation}
where \(d_{\alpha} = q_{\alpha} p + r_{\alpha}\) (\(0 \leq r_{\alpha} < p\)) and \(W_e\)
denotes the indecomposable \(H\)-representation of dimension \(e\).  Therefore,
by~\cite[Proposition~6.9]{Yasuda2014:p-Cyclic}, we find that the infinite sum
\begin{equation}
  \sum_{j \in \setN'} \nu(\GCov(D; -\infty, j)) \mtvL^{d - \vect{v}(j)}
  = \sum_{j \in \setN'} \nu(\GCov[H](D; j)) \mtvL^{d - \bm{v}'(j)}
\end{equation}
converges if and only if the inequality \(D_{W} \geq p\) holds, where \(D_{W}\) is the
invariant defined in~{\cite[Definition~6.8]{Yasuda2014:p-Cyclic}}.  By definition, we have
\begin{align}
  D_{W}
  &=
    \sum_{\alpha = 1}^{a} \left( r_{\alpha} \frac{(q_{\alpha} + 1) q_{\alpha}}{2}
    + (p - r_{\alpha}) \frac{q_{\alpha} (q_{\alpha} - 1)}{2} \right) \\
  &=
    \sum_{\alpha = 1}^{a} \left( \frac{q_{\alpha} p (q_{\alpha} - 1)}{2}
    + q_{\alpha} r_{\alpha} \right) \\
  &=
    B_{V}.
\end{align}
Therefore, the infinite sum~\cref{1st-sum} converges if and only if the inequality
\begin{equation}
  B_{V} \geq p \label{eq:cond-converge1}
\end{equation}
holds.

\subsubsection{Sum \cref{2nd-sum}}
Assume \(j_{0} \neq -\infty\).  Applying \cref{thm:v-function}, we have
\begin{equation}
  \bm{v}(j_{0}, j_{1}) =
  \begin{cases}
    \bm{e}_{V}^{>}(j_{0}) & \text{if \(pj_{0} > j_{1}\)}, \\
    \bm{e}_{V}^{<}(j_{0}, j_{1}) & \text{if \(pj_{0} < j_{1}\)}.
  \end{cases}
\end{equation}

Firstly, we shall consider the case \(pj_{0} > j_{1}\).  We shall compute the infinite sum
\begin{align}
  &
    \sum_{\substack{j_{0}, j_{1} \in \setN' \cup \{- \infty\}, \\ pj_{0} > j_{1}}}
    \nu(\GCov(D; j_{0}, j_{1})) \mtvL^{d - \bm{v}(j_{0}, j_{1})} \label{part:pj_0>j_1} \\
  &=
    \sum_{j_{0} \in \setN'} \nu(\GCov(D; j_{0}, - \infty)) \mtvL^{d - \bm{e}_{V}^{>}(j_{0})}
    + \sum_{\substack{j_{0}, j_{1} \in \setN', \\ pj_{0} > j_{1}}}
    \nu(\GCov(D; j_{0}, j_{1})) \mtvL^{d - \bm{e}_{V}^{>}(j_{0})}.
\end{align}
Since \(\dim \nu(\GCov(D; j_{0}, - \infty)) < \dim \nu(\GCov(D; j_{0}, j_{1}))\)
(\(j_{1} > - \infty\)), the first sum of the right hand side converges whenever the second sum
converges.  Thus it is enough to study the convergence of the second sum.  We have
\begin{align}
  &
    \sum_{\substack{j_{0}, j_{1} \in \setN', \\ pj_{0} > j_{1}}}
    \nu(\GCov(D; j_{0}, j_{1})) \mtvL^{d - \bm{e}_{V}^{>}(j_{0})} \\
  &=
    \sum_{\substack{j_{0}, j_{1} \in \setN', \\ pj_{0} > j_{1}}}
    {(\mtvL - 1)}^{2} \mtvL^{-2} \mtvL^{j_{0} - \floor{j_{0}/p} + j_{1} - \floor{j_{1}/p}}
    \mtvL^{d - \bm{e}_{V}^{>}(j_{0})} \\
  &= {(\mtvL - 1)}^{2} \mtvL^{d-2}
    \sum_{j_{0} \in \setN'}
      \mtvL^{j_{0} - \floor{j_{0}/p} - \bm{e}_{V}^{>}(j_{0})}
    \sum_{\substack{j_{1} = 1, \\ p \nmid j_{1}}}^{pj_{0} - 1}
      \mtvL^{j_{1} - \floor{j_{1}/p}}.
\end{align}
Since \(j_1-\floor{j_1/p}\) is an increasing function in \(j_{1}\),
\begin{equation}
  \dim \sum_{\substack{j_{1} = 1, \\ p \nmid j_{1}}}^{pj_{0} - 1}
  \mtvL^{j_{1} - \floor{j_{1} / p}} = (p j_{0} - 1)-\floor{(p j_{0} - 1) / p} = (p - 1) j_{0}.
\end{equation}
Therefore sum \cref{2nd-sum} converges if and only if
\begin{equation}
 j_{0} - \floor{j_{0} / p} - \bm{e}_{V}^{>}(j_{0}) + (p - 1) j_{0}
\end{equation}
tends to \(- \infty\) as \(j_{0}\) tends to \(\infty\).  From \cref{lem:reminder1}, this
function in \(j_0\) is equivalent to the following one modulo bounded functions:
\begin{equation}
  j_{0} - j_{0} / p - C_{V}^{>} j_{0} / p^{2} + (p - 1) j_{0} =
  p^{-2} j_{0} (p^{3} - p - C_{V}^{>}).
\end{equation}

We conclude:

\begin{proposition}
  The infinite sum~\cref{2nd-sum} converges if and only if the inequality
  \begin{equation}\label{eq:cond-converge2}
  p^{3} - p - C_{V}^{>} \leq -1
  \end{equation}
  holds.
\end{proposition}

\subsubsection{Sum \cref{3rd-sum}}
Sum \cref{3rd-sum} converges if and only if
\begin{equation}
  j_{0} - \floor{j_{0}/p} + j_{1} - \floor{j_{1}/p} - \bm{e}_{V}^{<}(j_{0}, j_{1}) \quad
  \left(
    = \dim \nu(\GCov(D; j_{0}, j_{1})) \mtvL^{d - \bm{v}(j_{0}, j_{1})} - d
  \right)
\end{equation}
tends to \(-\infty\) as \(j_0 + j_1\) tends to \(\infty\).  This function in \(j_0\) and
\(j_1\) is equivalent to the following one modulo bounded functions:
\begin{equation}
  f(j_{0}, j_{1}) \coloneq j_{0} \left(1 - p^{-1} -C^{<, 0} p^{-2} \right)
  + j_{1} \left(1 - p^{-1} - C^{<, 1} p^{-2} \right).
\end{equation}
We observe that \(f(j_0, j_1)\) lies between \(f(1, j_1)\) and \(f(\floor{j_{1} / p}, j_1)\).
Thus Sum \cref{3rd-sum} converges if and only if both \(f(1, j_1)\) and
\(f(\floor{j_{1} / p}, j_1)\) tend to \(-\infty\) as \(j_1\) tends to \(\infty\).  Since,
modulo bounded functions,
\begin{align}
  f(1,j_1) & \equiv j_1\left(1-p^{-1}-C^{<,1}p^{-2} \right)
\end{align}
and
\begin{align}
  f(\floor{j_{1} / p}, j_{1})
  &\equiv
    j_{1} p^{-1} \left(1 - p^{-1} - C^{<, 0} p^{-2} \right)
    + j_{1} \left(1 - p^{-1} - C_{V}^{<, 1} p^{-2} \right) \\
  &\equiv
    j_{1} \left(1 - p^{-2} - C_{V}^{<, 1} p^{-2} - C_{V}^{<, 0} p^{-3} \right),
\end{align}
we conclude:

\begin{proposition}
  The infinite sum~\cref{3rd-sum} converges if and only if the following inequalities hold:
  \begin{align}
  p^{2} - p - C_{V}^{<, 1} \leq -1, \label{eq:cond-converge3}\\
  p^{3} - p - C_{V}^{<, 0} - p C_{V}^{<, 1} \leq -1. \label{eq:cond-converge4}
  \end{align}
\end{proposition}

\subsubsection{Completing the proof of \cref{thm:cond-converge}}\label{sec:proof:cond-converge}
We found that the integral \(\int_{\GCov(D)} \mtvL^{d - \vect{v}}\) converges if and only
if inequalities \cref{eq:cond-converge1,eq:cond-converge2,eq:cond-converge3,eq:cond-converge4}
hold.  Since \(C_{V}^{>} = C_{V}^{<, 0} + p C_{V}^{<, 1}\), \cref{eq:cond-converge2} holds if
and only if \cref{eq:cond-converge4} does.  Similarly, since \(C_{V}^{<, 1} = p B_{V}\),
\cref{eq:cond-converge1} implies \cref{eq:cond-converge3}. Thus conditions
\cref{eq:cond-converge3} and \cref{eq:cond-converge4} are redundant and the theorem follows.

\begin{remark}\label{rem:not-lc:p=2}
  Consider the \(G\)-representation
  \(V = V_{1}^{\oplus x} \oplus V_{2}^{\oplus y} \oplus V_{3}\) with characteristic \(p = 2\).
  Note that when \(G=\setZ/p^2\setZ\), the \(G\)-representations with a pseudo-reflection are
  exactly the ones of this form (\cref{cor:pseudo-ref}).  Let \(\Delta\) be the divisor on
  \(X = V/G\) such that \(V \to (X, \Delta)\) is crepant.  From \cref{prop:boundary},
  \(\Delta \) is irreducible and reduced.  We see that the pair \( (X, \Delta)\) is log
  canonical if and only if \(y>0\).  Indeed, by direct computation, we get
  \begin{gather}
    B_{V} = 1,             \quad C_{V}^{>} = 5 + 2y, \\
    C_{V}^{<, 0} = 1 + 2y, \quad C_{V}^{<, 1} = 4.
  \end{gather}
  If \(y = 0\), then \(C_{V}^{>} = 5\) and hence \(p^{3} - p - C_{V}^{>} = 1 > 0\).  Therefore
  sum \cref{2nd-sum} has a term of dimension arbitrarily large.  From a variant of
  \cref{prop:Yasuda2019Prop2.1} for log pairs (see also \cite[Corollary
  1.4]{Yasuda2019:Motivic}), we can show that \((X, \Delta)\) is not log canonical.  Assume
  now that \(y > 0\).  Then we have
  \begin{gather}
    p^{3} - p - C_{V}^{>} = 1 - 2y \leq -1, \\
    p^{2} - p - C_{V}^{<, 1} = -2 \leq -1, \\
    p^{3} - p - C_{V}^{<, 0} - p C_{V}^{<, 1} = -3 -2y \leq -1,
  \end{gather}
  and hence sums \cref{2nd-sum} and \cref{3rd-sum} all converges.  From the equality
  \( B_{V} = p-1 \) and from \cite[Corollary 1.4]{Yasuda2019:Discrepancies}, we also see that
  sum \cref{1st-sum} has terms of dimensions bounded above.  Again from the variant of
  \cref{prop:Yasuda2019Prop2.1}, \((X, \Delta)\) is log canonical.
\end{remark}

\subsection{Evaluation of discrepancies}
Using \cref{prop:Yasuda2019Prop2.1}, we can evaluate \(d(X)\) of the quotient
\(X = \schA_{k}^{d}/G\).  As in the paragraph after \cref{prop:Yasuda2019Prop2.1}, let
\(M_{\vect{j}}\) be a stratum of \(\spcJ X\) corresponding a stratum \(\GCov(D; \vect{j})\) of
\(\GCov(D)\).  Let us compute
\(\lambda(M_{\vect{j}}) = \dim \nu(\GCov(D; j_{0}, j_{1})) + d - \bm{v}(j_{0}, j_{1})\).

\subsubsection{The case \(j_{0} = - \infty\)}
From arguments in \cref{sec:1st-sum}, we have
\begin{align}
  \lambda(M_{(- \infty, j)})
  &=
    \dim \nu(\GCov(D; - \infty, j)) + d - \bm{v}(- \infty, j) \\
  &=
    \dim \nu(\GCov[H](D; j)) + d - \bm{v}'(j) \\
  &=
    \lambda(N_{j}),
\end{align}
where \(N_{j}\) is defined in the same way as \(M_{\vect{j}}\) for \(\schA_{k}^{d} / H\).
From~\cite[Equation~(3.1), (3.2)]{Yasuda2019:Discrepancies}, if \(p - 1 - B_{V} \leq 0\), then
we have
\begin{align}
  \sup_{j_{0} = - \infty} \lambda(M_{\vect{j}})
  &=
    b + \max_{1 \leq l_{1} < p} \{l_{1} - \sht_{W}(l_{1})\} \label{eq:sup:j0=-infty} \\
  &=
    d - B_{V} + \max_{1 \leq l_{1} < p}\{\sht_{W}(p - l_{1}) + l_{1}\},
\end{align}
where \(b\) denotes the number of the indecomposable direct summands of the induced
\(H\)-representation \(W\) and \(\sht_{W}\) the function associated to \(W\) defined as
follows.  For an indecomposable representation \(W_{e}\) of dimension \(e\), we define
\begin{equation}
  \sht_{W_{e}}(l) \coloneq
  \sum_{i = 1}^{e - 1} \floor{\frac{il}{p}}.
\end{equation}
In general, for the case \(W = \bigoplus_{e} W_{e}\), we define
\(\sht_{W} \coloneq \sum_{e} \sht_{W_{e}}\).  As for the value of \(b\), from
\cref{lem:Jordan-form-of-p-power}, we have
\begin{equation}
  b = \sum_{d_{\alpha} < p} d_{\alpha} + \sum_{d_{\alpha} \geq p} p.
\end{equation}
Similarly, as for the value of \(\sht_{W}(l)\), from \cref{eq:induced-subrepresentation}, we
have
\begin{equation}
  \sht_{W}(l) = \sum_{\alpha = 1}^{a}
  \left(
    r_{\alpha} \sum_{i = 1}^{q_{\alpha}} \floor{\frac{il}{p}}
    + (p - r_{\alpha}) \sum_{i = 1}^{q_{\alpha} - 1} \floor{\frac{il}{p}}
  \right).
\end{equation}
If \(p - 1 - B_{V} > 0\), then we have
\(\sup_{j_{0} = - \infty} \lambda(M_{\vect{j}})= \infty\).

\subsubsection{The case \(j_{0} \neq - \infty\)}
If \(j_{1} = - \infty\), then we have
\begin{equation}
  \GCov(D; j_{0}, - \infty) = \schG_{m} \times \schA^{j_{0} - 1 - \floor{j_{0}/p}}.
\end{equation}
Since \(\bm{v}(j_{0}, - \infty) = \bm{e}_{V}^{>}\) depends only on \(j_{0}\), thus
\(\lambda(M_{(j_{0}, - \infty)}) < \lambda(M_{(j_{0}, j_{1})})\) for
\(- \infty < j_{1} < pj_{0}\).  Therefore, we may assume that \(j_{1} \neq - \infty\) to
evaluate \(\sup \lambda(M_{\vect{j}})\).  Assuming \(j_{0}, j_{1} \neq - \infty\), if we write
\(j_{i} = n_{i} p^{2} + m_{i} p + l_{i}\) (\(i = 1, 2\); \(0 \leq m_{i} < p\),
\(1 \leq l_{i} < p\)), then we have
\begin{align}
  \lambda(M_{\vect{j}})
  &= d + j_{0} - \floor{j_{0}/p} + j_{1} - \floor{j_{1}/p} - \vect{v}(j_{0}, j_{1}) \\
  &= d + (p^{2} - p) n_{0} + (p - 1) m_{0} + l_{0}
    + (p^{2} - p) n_{1} + (p - 1) m_{1} + l_{1}
    - \vect{v}(j_{0}, j_{1}).
\end{align}

Firstly, we consider the case \(p j_{0} > j_{1}\).

\begin{lemma}\label{lem:sup:pj0>j1}
  We have
  \begin{equation}\label{eq:sup:pj0>j1}
  \sup_{pj_{0} > j_{1}} \lambda(M_{\vect{j}}) =
  \begin{dcases}
    d + \max \{%
    (p^{2} - 1) m_{0} + l_{0} p - \bm{e}_{V}^{>}(m_{0} p + l_{0})%
    \} & \text{if \(C_{V}^{>} \geq p^{3} - p\)}, \\
    \infty & \text{otherwise}.
  \end{dcases}
\end{equation}
\end{lemma}
\begin{proof}
  Let \(pj_{0} > j_{1} \neq - \infty\).  The maximum value of \(\lambda(M_{\vect{j}})\)
  (\(pj_{0} > j_{1}\)) is given when \(j_{1} = pj_{0} - 1\), equivalently,
  \(n_{1} = n_{0} p + m_{0}\), \(m_{1} = l_{0} - 1\) and \(l_{1} = p - 1\).  By
  \cref{lem:reminder1}, we get
  \begin{equation}
    \lambda(M_{\vect{j}})
    = d + (p^{3} - p - C_{V}^{>}) n_{0} + (p^{2} - 1) m_{0} + l_{0} p
    - \bm{e}_{V}^{>}(m_{0}p + l_{0}).
  \end{equation}
  If \(p^{3} - p - C_{V}^{>} > 0\), then
  \(\sup_{pj_{0} > j_{1}} \lambda(M_{\vect{j}}) = \infty\).  Otherwise, we have
  \begin{equation}
    \sup_{pj_{0} > j_{1}}(\lambda(M_{\vect{j}})) =
    d + \max_{0 \leq m_{0}, m_{1} < p; 1 \leq l_{0}, l_{1} < p} \{%
    (p^{2} - 1) m_{0} + l_{0} p - \bm{e}_{V}^{>}(m_{0} p + l_{0})%
    \},
  \end{equation}
  which completes the proof.
\end{proof}

Secondly, we consider the case \(p j_{0} < j_{1}\).

\begin{lemma}\label{lem:sup:pj0<j1}
  If \(C_{V}^{<, 0} \geq p^{2} - p\) and \(C_{V}^{<, 1} \geq p^{2} - p\), then we have
  \begin{equation}\label{eq:sup:pj0<j1}
    \sup_{pj_{0} < j_{1}} \lambda(M_{\vect{j}}) =
    d + \max_{\substack{0 \leq m_{0}, m_{1} < p, \\ 1 \leq l_{0}, l_{1} < p}} \{
    (p - 1) m_{0} + l_{0} + (p - 1) m_{1} + l_{1}
    - \bm{e}_{V}^{<}(m_{0}p + l_{0}, m_{1}p + l_{1})
    \}.
  \end{equation}
  If \(C_{V}^{<, 1} < p^{2} - p\), then
  \(\sup_{pj_{0} < j_{1}} \lambda(M_{\vect{j}}) = \infty\).
\end{lemma}
\begin{proof}
  By \cref{lem:reminder1}, we get
  \begin{multline}\notag
    \lambda(M_{\vect{j}}) = d
    + (p^{2} - p - C_{V}^{<, 0}) n_{0} + (p - 1) m_{0} + l_{0} \\
    + (p^{2} - p - C_{V}^{<, 1}) n_{1} + (p - 1) m_{1} + l_{1} - \bm{e}_{V}^{<}(m_{0} p +
    l_{0}, m_{1} p + l_{1}).
  \end{multline}
  Assuming that \(C_{V}^{<, 0} \geq p^{2} - p\) and \(C_{V}^{<, 1} \geq p^{2} - p\), we get
  the first assertion.  It is obvious that
  \(\sup_{pj_{0} < j_{1}} \lambda(M_{\vect{j}}) = \infty\) if \(C_{V}^{<, 1} < p^{2} - p\).
\end{proof}

\begin{theorem}\label{thm:sup-of-dimension}
  If \(B_{V} \geq p - 1\) and \(C_{V}^{>} \geq p^{3} - p\), then we have
  \(\sup_{\vect{j}} \lambda(M_{\vect{j}}) < \infty\) and
  \begin{equation}
    \sup_{\vect{j}} \lambda(M_{\vect{j}}) = \max \left\{
      \sup_{j_{0} = - \infty} \lambda(M_{\vect{j}}),
      \sup_{pj_{0} > j_{1}} \lambda(M_{\vect{j}}),
      \sup_{pj_{0} < j_{1}} \lambda(M_{\vect{j}}) \right\},
  \end{equation}
  where the suprema on the right hand side are given by formulae
  \cref{eq:sup:j0=-infty,eq:sup:pj0>j1,eq:sup:pj0<j1}.  Conversely, if \(B_{V} < p - 1\) or
  \(C_{V}^{>} < p^{3} - p\), then \(\sup_{\vect{j}} \lambda(M_{\vect{j}}) = \infty\).
\end{theorem}
\begin{proof}
  Since \(C_{V}^{<, 1} = p B_{V}\), thus \(B_{V} \geq p - 1\) implies that
  \(C_{V}^{<, 1} \geq p^{2} - p\).  Therefore, it is enough to show that
  \(C_{V}^{<, 0} \geq p^{2} - p\).  Assume that \(C_{V}^{<, 0} < p^{2} - p\).  Then, the lemma
  below shows that the \(G\)-representation \(V\) is of the form
  \(V_{1}^{\oplus x} \oplus V_{3}\) (\(p = 2\)).  However, if this is the case, we have
  \(C_{V}^{>} = 5 < 6 = p^{3} - p\), which contradicts to the assumption
  \(C_{V}^{>} \geq p^{3} - p\).

  Consequently, if \(B_{V} \geq p - 1\) and \(C_{V}^{>} \geq p^{3} - p\), then the suprema
  \(\sup_{j_{0} = - \infty} \lambda(M_{\vect{j}})\),
  \(\sup_{pj_{0} > j_{1}} \lambda(M_{\vect{j}})\) and
  \(\sup_{pj_{0} < j_{1}} \lambda(M_{\vect{j}})\) are all finite and they are given by
  \cref{eq:sup:j0=-infty,eq:sup:pj0>j1,eq:sup:pj0<j1}.  The converse is obvious.
\end{proof}

\begin{lemma}\label{prop:lower-bounds-of-CV<0}
  Let \(V\) be a faithful \(G\)-representation.  We have the following.
  \begin{enumerate}
  \item The inequality \(C_{V}^{<, 0} \geq p^{2}\) holds if \(p \geq 5\).
  \item The inequality \(C_{V}^{<, 0} \geq p^{2} - p\) holds except if the
    \(G\)-representation \(V\) is of the form
    \(V_{1}^{\oplus x} \oplus V_{3}\) (\(p = 2\)).
  \end{enumerate}
\end{lemma}
\begin{proof}
  (1).  Without loss of generality, we may assume \(V\) is indecomposable of dimension \(d\)
  (\(p + 1 \leq d \leq p^{2}\)).  By \cref{lem:bounds-for-A&B}, we get
  \begin{align}
    \label{eq:1}
    C_{V}^{<, 0} - p^{2}
    &=
      p A_{d} - (p - 1) B_{d} - p^{2} \\
    &\geq
      p \left(\frac{(p - 1) d}{2} - \frac{p^{2}}{8}\right)
      - (p - 1) \left(\frac{d (d - p)}{2 p} + \frac{p}{8}\right) - p^{2}.
  \end{align}
  We denote by \(\Phi(d)\) the last expression above.  Let us show that \(\Phi(d) \geq 0\).
  By direct computation, we have
  \begin{equation}
    \Phi(d) =
    - \frac{p - 1}{2p} d^{2} + \frac{p^{2} - 1}{2} d - \frac{p^{3} + 9 p^{2} - p}{8},
  \end{equation}
  hence \(\Phi(d)\) is upward-convex with \(d\) regarded as a real variable.  It is enough to
  check the values \(\Phi(p + 1)\) and \(\Phi(p^{2})\) are both non-negative.  We get
  \begin{align}
    \Phi(p + 1)
    &=
      \frac{p^{2} (3 p (p - 3) - 7) + 4}{8 p} \\
    \Phi(p^{2})
    &=
      \frac{p (3 p (p - 5) + 2 p + 1)}{8}.
  \end{align}
  It is obvious that \(\Phi(p + 1) \geq 0\) and \(\Phi(p^{2}) \geq 0\) if \(p \geq 5\).
  Consequently, we have \(C_{V}^{<, 0} - p^{2} \geq \Phi(d) \geq 0\).

  (2).  We may assume \(p \leq 3\).  Firstly, we consider the case that \(V\) is
  indecomposable of dimension \(d \geq p + 1\).  When \(p = 3\), \(d\) varies from \(4\) to
  \(9\).  Checking each case by direct computation, we get
  \begin{gather}
    C_{V_{4}}^{<, 0} = 7, \quad C_{V_{5}}^{<, 0} = 8, \quad C_{V_{6}}^{<, 0} = 12, \\
    C_{V_{7}}^{<, 0} = 8, \quad C_{V_{8}}^{<, 0} = 7, \quad C_{V_{9}}^{<, 0} = 9,
  \end{gather}
  and hence \(C_{V}^{<, 0} \geq p^{2} - p\) holds.

  If \(p = 2\), we have
  \begin{equation}
    C_{V_{1}}^{<, 0} = 0, \quad  C_{V_{2}}^{<, 0} =  2, \quad
    C_{V_{3}}^{<, 0} = 1, \quad C_{V_{4}}^{<, 0} = 2,
  \end{equation}
  and hence the proof is completed.
\end{proof}

\begin{example}\label{ex:discrep-in-char3}
  Let \(p = 3\) and \(V\) the indecomposable \(G = \setZ/p^{2}\setZ\)-representation of
  dimension \(d\) (\(p + 1 < d \leq p^{2}\)).  Then, according to computations with
  Sage~\cite{SageMath}, we get the following:

  \begin{table}[htb]
    \centering
    \begin{tabular}{ccccc}
      \toprule \(d = \dim V\)
      & \(\sup_{j_{0} = - \infty} \lambda(M_{\vect{j}})\)
      & \(\sup_{pj_{0} > j_{1}} \lambda(M_{j})\)
      & \(\sup_{pj_{0} < j_{1}} \lambda(M_{j})\)
      & \(d(X)\) \\
      \cmidrule(r){1-1} \cmidrule(l){2-5}
      \(4\) & \(\infty\) & \(\infty\) & \(\infty\) &        \\
      \(5\) & \(5\)      & \(3\)      & \(4\)      & \(-1\) \\
      \(6\) & \(5\)      & \(2\)      & \(3\)      & \(0\)  \\
      \(7\) & \(4\)      & \(1\)      & \(4\)      & \(2\)  \\
      \(8\) & \(4\)      & \(0\)      & \(4\)      & \(3\)  \\
      \(9\) & \(4\)      & \(-2\)     & \(3\)      & \(4\)  \\
      \bottomrule
    \end{tabular}
    \caption{discrepancies in characteristic \(3\)}
  \end{table}
\end{example}

\subsection{Upper and lower bounds}
We shall give lower bounds of \(\bm{e}_{V}^{>}\) and \(\bm{e}_{V}^{<}\) and apply them to
determine when the quotient variety \(X\) is terminal, canonical or log canonical under the
condition that the given \(G\)-representation is indecomposable.

\begin{lemma}\label{prop:lower-bounds-of-e}
  We have
  \begin{equation}
   \bm{e}_{V}^{>}(j_{0})\geq \frac{j_{0} C_{V}^{>}}{p^{2}}, \quad
   \bm{e}_{V}^{<}(j_{0}, j_{1}) \geq \frac{j_{0} C_{V}^{<, 0} + j_{1} C_{V}^{<, 1}}{p^{2}}.
  \end{equation}
\end{lemma}
\begin{proof}
  By definition, we get
  \begin{align}
    \bm{e}_{d_{\alpha}}^{>}(j_{0})
    &=
      \sum_{\substack{0 \leq i_{0} + i_{1}p < d_{\alpha}, \\ 0 \leq i_{0}, i_{1} < p}}
      \ceil{\frac{p i_{0} j_{0} + (p^{2} - p + 1) i_{1} j_{0}}{p^{2}}} \\
    &\geq
      \sum_{\substack{0 \leq i_{0} + i_{1}p < d_{\alpha}, \\ 0 \leq i_{0}, i_{1} < p}}
      \frac{p i_{0} j_{0} + (p^{2} - p + 1) i_{1} j_{0}}{p^{2}} \\
    &=
      \frac{j_{0}}{p^{2}} \left(
      p
      \sum_{\substack{0 \leq i_{0} + i_{1}p < d_{\alpha}, \\ 0 \leq i_{0}, i_{1} < p}}
      i_{0}
      + (p^{2} - p + 1)
      \sum_{\substack{0 \leq i_{0} + i_{1}p < d_{\alpha}, \\ 0 \leq i_{0}, i_{1} < p}}
      i_{1}
      \right) \\
    &= \frac{j_{0} C_{d_{\alpha}}^{>}}{p^{2}}.
  \end{align}
  Taking sum over \(\alpha\), we get the first inequality.

  Similarly, we have
  \begin{align}
    \bm{e}_{d_{\alpha}}^{<}(j_{0}, j_{1})
    &=
      \sum_{\substack{0 \leq i_{0} + i_{1}p < d_{\alpha}, \\ 0 \leq i_{0}, i_{1} < p}}
      \ceil{\frac{p i_{0} j_{0} - (p - 1) i_{1} j_{0} + p i_{1} j_{1}}{p^{2}}} \\
    &\geq
      \sum_{\substack{0 \leq i_{0} + i_{1}p < d_{\alpha}, \\ 0 \leq i_{0}, i_{1} < p}}
      \frac{p i_{0} j_{0} - (p - 1) i_{1} j_{0} + p i_{1} j_{1}}{p^{2}} \\
    &=
      \frac{j_{0}}{p^{2}} \left(
      p
      \sum_{\substack{0 \leq i_{0} + i_{1}p < d_{\alpha}, \\ 0 \leq i_{0}, i_{1} < p}}
      i_{0}
      - (p - 1)
      \sum_{\substack{0 \leq i_{0} + i_{1}p < d_{\alpha}, \\ 0 \leq i_{0}, i_{1} < p}}
      i_{1}\right)
      + \frac{j_{1}}{p^{2}} p
      \sum_{\substack{0 \leq i_{0} + i_{1}p < d_{\alpha}, \\ 0 \leq i_{0}, i_{1} < p}}
      i_{1} \\
    &=
      \frac{j_{0} C_{d_{\alpha}}^{<, 0} + j_{1} C_{d_{\alpha}}^{<, 1}}{p^{2}},
  \end{align}
  and hence we get the second inequality.
\end{proof}

We can now give the upper and bound of \(\sup \lambda(M_{\vect{j}})\).  For the \(j_{0} = -
\infty\) part, from~\cite[Theorem~1.2]{Yasuda2019:Discrepancies}, we have the following
bounds:

\begin{lemma}[{\cite{Yasuda2019:Discrepancies}}]
  We have
  \begin{equation}\label{ineq:lower-bound:j0=-infty}
    \sup_{j_{0} = - \infty} \lambda(M_{\vect{j}}) \geq d + p - 1 - B_{V}.
  \end{equation}
  Furthermore, if \(B_{V} \geq p\), then we have
  \begin{equation}\label{ineq:upper-bound:j0=-infty}
    \sup_{j_{0} = - \infty} \lambda(M_{\vect{j}}) \leq d + 1 - \frac{2 B_{V}}{p}.
  \end{equation}
\end{lemma}

\begin{proposition}\label{thm:upper-bound-of-sup}
  Assume that \(B_{V} \geq p\), \(C_{V}^{>} \geq p^{3}\) and \(C_{V}^{<, 0} \geq p^{2}\).
  Then, we have
  \begin{equation}
    \sup \lambda(M_{\vect{j}}) \leq \max
    \left\{
      d + 1 - \frac{2 B_{V}}{p},
      d + p - \frac{C_{V}^{>}}{p^{2}},
      d + 2 - \frac{C_{V}^{<, 0} + C_{V}^{<, 1}}{p^{2}}
    \right\}.
  \end{equation}
\end{proposition}
\begin{proof}
  We shall first consider \(\sup_{pj_{0} > j_{1}} \lambda(M_{\vect{j}})\).  By
  \cref{prop:lower-bounds-of-e}, we have
  \begin{equation}
    (p^{2} - 1) m_{0} + l_{0} p - \bm{e}_{V}^{>}(m_{0} p + l_{0}) \leq
    \left(p^{2} - 1 - \frac{C_{V}^{>}}{p}\right) m_{0}
    + \left(p - \frac{C_{V}^{>}}{p^{2}}\right) l_{0}.
  \end{equation}
  Since we assume \(C_{V}^{>} \geq p^{3}\), thus the coefficients of \(m_{0}\) and \(l_{0}\)
  in the right hand side are non-positive.  Therefore, the right hand side attains the maximum
  at \(m_{0} = 0\) and \(l_{0} = 1\).  From \cref{lem:sup:pj0>j1}, we get
  \begin{equation}\label{ineq:upper-bound:pj0>j1}
    \sup_{pj_{0} > j_{1}} \lambda(M_{\vect{j}}) \leq d + p - \frac{C_{V}^{>}}{p^{2}}.
  \end{equation}

  Next, we shall consider \(\sup_{pj_{0} < j_{1}} \lambda(M_{\vect{j}})\).  By
  \cref{prop:lower-bounds-of-e}, we have
  \begin{multline}\notag
    (p - 1) m_{0} + l_{0} + (p - 1) m_{1} + l_{1} - \bm{e}_{V}^{<}(m_{0}p + l_{0}, m_{1}p + l_{1}) \\
    \leq
    \left(p - 1 - \frac{C_{V}^{<, 0}}{p}\right) m_{0}
    + \left(1 - \frac{C_{V}^{<, 0}}{p^{2}}\right) l_{0}
    + \left(p - 1 - \frac{C_{V}^{<, 1}}{p}\right) m_{1}
    + \left(1 - \frac{C_{V}^{<, 1}}{p^{2}}\right) l_{1}.
  \end{multline}
  Note that the assumption \(B_{V} \geq p\) implies \(C_{V}^{<, 1} \geq p^{2}\).  Since we
  have \(C_{V}^{<, 0} \geq p^{2}\) and \(C_{V}^{<, 1} \geq p^{2}\), thus the coefficients in
  the last expression are non-positive.  Therefore, the last expression takes the maximum at
  \(m_{0} = m_{1} = 0\) and \(l_{0} = l_{1} = 1\).  From \cref{lem:sup:pj0<j1}, we get
  \begin{equation}\label{ineq:upper-bound:pj0<j1}
    \sup_{pj_{0} < j_{1}} \lambda(M_{\vect{j}}) \leq
    d + 2 - \frac{C_{V}^{<, 0} + C_{V}^{<, 1}}{p^{2}}.
  \end{equation}

  Combining \cref{ineq:upper-bound:j0=-infty,ineq:upper-bound:pj0>j1,ineq:upper-bound:pj0<j1},
  we get the claim.
\end{proof}

As a conclusion of this section, we get the following.

\begin{theorem}\label{thm:estimate-discrep}
  Assume that \(V = V_{d}\) is an indecomposable \(G\)-representation of dimension \(d\)
  (\(p + 1 < d \leq p^{2}\)) (with this assumption, \(V\) has no pseudo-reflection and
  \(V \to X \coloneq V/G\) is crepant).  Then,
  \begin{equation}
    \text{\(X\) is }
    \begin{cases}
      \text{terminal}, \\
      \text{canonical}, \\
      \text{log canonical}, \\
      \text{not log canonical}
    \end{cases}
    \text{if and only if }
    \begin{cases}
      d \geq 2p + 1, \\
      d \geq 2p, \\
      d \geq 2p - 1, \\
      d < 2p - 1.
    \end{cases}
  \end{equation}
\end{theorem}
\begin{proof}
  First, we consider the case \(d < 2p - 1\).  From the definition of \(B_{V}\), we get
  \(B_{V} < p - 1\).  By \cref{thm:sup-of-dimension}, we get
  \(\sup \lambda(M_{\vect{j}}) = \infty\) and hence \(d(X) = - \infty\).

  Next, we consider the case \(d = 2p - 1\).  Since we also assume \(d > p + 1\), thus we have
  \(p \geq 3\).  By direct computation, we have \(B_{V} = p - 1\) and
  \(C_{V}^{>} = 2p^{3} - 4p^{2} + 3p - 1 > p^{3} - p\) and hence \(d(X) > - \infty\) by
  \cref{thm:sup-of-dimension}.  We remark that by \cref{ineq:lower-bound:j0=-infty} we have
  \begin{equation}
    \sup \lambda(M_{\vect{j}})
    \geq \sup_{j_{0} = - \infty} \lambda(M_{\vect{j}})
    \geq d + p - 1 - B_{V}.
  \end{equation}
  and hence \(d(X) \leq -1\) by \cref{prop:Yasuda2019Prop2.1}.  Therefore, we get
  \(d(X) = -1\).

  Thirdly, we consider the case \(d = 2p\).  Then, we have
  \begin{align}
    B_{V} &= p, \\
    C_{V}^{>} &= 2 p^{3} - 2 p^{2} + p \geq p^{3} \geq p^{3} - p + 1.
  \end{align}
  and hence by \cref{prop:cond-klt}, the quotient \(X = V/G\) is klt and \(d(X) \geq 0\)
  (recall that since \(\omega_{X}\) is invertible, thus \(d(X) > -1\) implies
  \(d(X) \geq 0\)).  On the other hand, by \cref{ineq:lower-bound:j0=-infty}, we have
  \begin{equation}
    \sup \lambda(M_{\vect{j}})
    \geq \sup_{j_{0} = - \infty} \lambda(M_{\vect{j}})
    \geq d + p - 1 - B_{V}
    = 2 p - 1,
  \end{equation}
  and hence \(d(X) \leq 0\).  Thus we get \(d(X) = 0\).

  Finally, we consider the case \(d \geq 2p + 1\).  When \(p = 3\), the assertion follows from
  \cref{ex:discrep-in-char3}.  We assume that \(p \geq 5\).  We remark that \(A_{d}\) and
  \(B_{d}\) are monotonically increasing function in \(d\), so are \(C_{V}^{>}\) and
  \(C_{V}^{<, 0} + C_{V}^{<, 1}\).  From \cref{thm:upper-bound-of-sup}, it is enough to show
  \begin{equation}\label{eq:key:estimate-discrep}
    \max \left\{
      d + 1 - \frac{2 B_{V}}{p},
      d + p - \frac{C_{V}^{>}}{p^{2}},
      d + 2 - \frac{C_{V}^{<, 0} + C_{V}^{<, 1}}{p^{2}}
    \right\} < d - 1
  \end{equation}
  in the case \(d = 2p + 1\).  In this case, we have
  \begin{align}
    B_{V} &= p + 2, \\
    C_{V}^{>} &= 2p^{3} - p + 2, \\
    C_{V}^{<, 0} + C_{V}^{<, 1} &= p^{3} - p^{2} + p + 2,
  \end{align}
  and hence
  \begin{align}
    d + 1 - \frac{2 B_{V}}{p}
    &=
      d - 1 - \frac{4}{p}, \\
    d + p - \frac{C_{V}^{>}}{p^{2}}
    &=
      d - 1 + 1 + p - \frac{2 p^{3} - p + 2}{p^{2}} \\
    &=
      d - 1 + 1 - p - \frac{p - 2}{p^{2}}, \\
    d + 2 - \frac{C_{V}^{<, 0} + C_{V}^{<, 1}}{p^{2}}
    &=
      d - 1 + 3 - \frac{p^{3} - p^{2} + p + 2}{p^{2}} \\
    &=
      d - 1 + 4 - p - \frac{p + 2}{p^{2}}.
  \end{align}
  Thus \cref{eq:key:estimate-discrep} holds.
\end{proof}

\bibliography{refs}
\bibliographystyle{amsplain}
\end{document}